\documentclass[12pt]{article}
\usepackage{amsmath, amsfonts, amsthm, amssymb, color, hyperref, extarrows, enumitem, nicefrac,blindtext, mathtools, cite, anyfontsize, comment, appendix}
\usepackage{tikz}
\usetikzlibrary{decorations.pathreplacing, arrows.meta}
\usepackage{placeins}

\newtheorem{theorem}{Theorem}[section]
\newtheorem{definition}[theorem]{Definition}

\newtheorem{lemma}[theorem]{Lemma}
\newtheorem{proposition}[theorem]{Proposition}

\numberwithin{equation}{section}

\newtheorem{remark}[theorem]{Remark}

\newtheorem{example}{Example}

\usepackage[hyperpageref]{backref}

\usepackage{cite}

\hypersetup{
	colorlinks   = true,
	citecolor    = magenta}
	
\begin{document}

\title{\textbf{\Large Constructing bounded pseudoconvex domains of finite D'Angelo type with prescribed weak loci}}
\author{Martino Fassina, \ \ Yifei Pan\ \  and \ \ Yuan Zhang}
\date{}

\maketitle

\begin{abstract}
We construct  smoothly  bounded pseudoconvex or convex domains in $\mathbb{C}^n$ whose weakly pseudoconvex loci realize    prescribed closed sets.
\end{abstract}

\renewcommand{\thefootnote}{\fnsymbol{footnote}}
\footnotetext{\hspace*{-7mm}
\begin{tabular}{@{}r@{}p{16.5cm}@{}}
& 2020 Mathematics Subject Classification. Primary  32T27; Secondary 32T25. \\
& Key words and phrases. Weakly pseudoconvex locus, D'Angelo type, 
Levi determinant,  convexity. 
\end{tabular}}

\section{Introduction}

A fundamental problem in several complex variables and complex geometry is to understand the geometry and analysis of weakly pseudoconvex boundary points of pseudoconvex domains. One of the central invariants in this direction is D'Angelo type, which measures the maximal order of contact between the boundary and germs of holomorphic curves passing through a given boundary point.

For  smooth hypersurfaces locally defined by real-analytic functions, the structure of points of infinite D'Angelo type is rigidly constrained. A landmark theorem of D'Angelo \cite{DA82} states that  the D'Angelo type  at a   point of such a hypersurface  is infinite if and only if the hypersurface  contains a  germ of a nonconstant  holomorphic curve through that point. On the other hand, according to a global theorem of Diederich--Forn\ae ss \cite{DF78}, every compact real-analytic variety in $\mathbb C^n$  contains no such germ. Combining these two facts, one obtains that every smoothly bounded pseudoconvex domain with real-analytic boundary is of finite D'Angelo type at every boundary point.     %Thus, in the real-analytic category, weakly pseudoconvex points may occur, but they cannot be of infinite type; equivalently, the boundary cannot contain germs of holomorphic curves.

  While real-analyticity forces infinite-type points to arise from complex analytic varieties, this characterization fails for general smooth hypersurfaces. Nevertheless, D'Angelo proved in \cite%[Theorem 4.11]
  {DA82} that the set of points of finite  type is open on the hypersurfaces. In particular, the condition that the  D'Angelo type is equal to  2 corresponds to  Levi non-degeneracy   and hence is    open. For pseudoconvex domains, this is precisely strict pseudoconvexity. The weakly pseudoconvex locus is the set of boundary points at which strict
pseudoconvexity fails. This locus is important not only geometrically  but
also analytically: in the $\bar\partial$-Neumann problem, the availability
of local subelliptic estimates depends sensitively on its size, flatness, and
structure. The role of the weakly pseudoconvex locus is also evident in the celebrated Diederich--Forn\ae ss worm domain \cite{DF76}, which has served as a fundamental example permeating the field of several complex variables and showing that smooth pseudoconvexity alone does not guarantee the expected global analytic behavior. Its weakly pseudoconvex boundary structure is closely tied to phenomena such as the failure of Stein neighborhood bases and the irregularity of canonical operators.

In a related work, the second author, jointly with Shao, Shi and Wu \cite{PSSW}, has shown that, given a smoothly bounded real-analytic pseudoconvex domain $\Omega$, there exists a smooth  perturbation  $\Omega_\epsilon$ of $\Omega$ whose weakly pseudoconvex locus coincides with that of $\Omega$.  Moreover, $\Omega_\epsilon$ is strictly pseudoconvex and real-analytic away from this locus. At each point of this locus, the boundary of $\Omega_\epsilon$ fails to be real-analytic, while the corresponding D'Angelo type coincides with that of the original domain.
%that is strictly pseudoconvex and real-analytic away from  the weakly pseudoconvex locus of $\Omega$. %so that its weakly pseudoconvex locus is identical to that of $\Omega$. 
%Moreover,   at each point of this  locus, the boundary of $\Omega_\epsilon$ fails to be real-analytic, while  the corresponding  D'Angelo type  coincides with that of the original domain $\Omega$.  %Namely,  a pre-existing weakly pseudoconvex locus in the real-analytic setting can be preserved under a suitable flat perturbation,  with no change in its finite-type geometry.

Unlike  the rigid  analytic structure of weakly pseudoconvex loci in the real-analytic category, smooth pseudoconvex boundaries may exhibit substantially greater flexibility.  
 The purpose of this paper is to address a complementary realization phenomenon:  in the smooth category, weakly pseudoconvex loci can be
prescribed a priori with considerable freedom, including sets that need not arise from real-analytic geometry,  while retaining finite   type control. %More precisely, we construct smoothly bounded  pseudoconvex domains of finite D'Angelo type whose weak loci realize arbitrary compact subsets of \(\mathbb R^{n-1}\subset \mathbb C^{n-1}\). %Such sets include, for instance, arbitrary Cantor sets in $\mathbb R^{n-1}$. 
Thus, even under a uniform finite-type bound,  weak loci may have  exotic  non-analytic structures. Throughout the paper, we write points in \(\mathbb C^n\) as
\((z,w)\in \mathbb C^{n-1}\times \mathbb C\).

The first   theorem  constructs a smooth convex domain of finite type  in $\mathbb C^2$ whose weakly pseudoconvex locus has the structure of a family  of tree-like sets in the   $\operatorname{Im } w $ (i.e., the Reeb) direction, rooted at the points of a  prescribed compact set $E$ in the  $z$ (complex tangential) direction. 

\begin{theorem}\label{main1}
Let \(E\subset \mathbb R\subset \mathbb C\) be an arbitrary compact set. For every integer \(m\ge 2\), there exist  a smoothly bounded   convex domain \(\Omega\subset \mathbb C^2\) and a neighborhood $\mathcal U$   of  the set $ (E+i0)\times (\mathbb R+i0)
$ in $\mathbb C^2$,   such that the weakly pseudoconvex locus of $\Omega$ is precisely
\[
W=\{(z,w)\in b\Omega:\operatorname{Re} z\in E,\  \operatorname{Im} z=0\}\cap \  \overline{\mathcal U}.
\]
Moreover,  the   D'Angelo type is equal to $2m$ at every point  in $W$.
\end{theorem}

By adding a controlling term in  the $\operatorname{Im } w $ direction to  the defining equation, we shall   eliminate the fibers in that direction and  obtain a smoothly bounded convex domain whose weakly pseudoconvex locus is precisely the prescribed compact set. 

\begin{theorem}\label{main1'}
Let \(E\subset \mathbb R\subset \mathbb C\) be an arbitrary  compact set. For every integer \(m\ge 2\), there exists a smoothly bounded   convex domain \(\Omega\subset \mathbb C^2\),  such that the weakly pseudoconvex locus of $\Omega$ is precisely
\[
W=\{(z,w)\in b\Omega:\operatorname{Re} z\in E,\  \operatorname{Im} z=0\ \text{and}\  \operatorname{Im} w=0 \} \cong_{\text{diff}} E.
\]
Moreover,  the   D'Angelo type is equal to $2m$ at every point in $W$.
\end{theorem}

A surprising special case occurs when $E$ is a Cantor set, which is totally disconnected but may have either zero or  positive Lebesgue measure. More generally, depending on the construction, a Cantor set in $\mathbb R$ may have any prescribed Hausdorff dimension between $0$ and $1$. In this setting,  the construction in Theorem \ref{main1'} can realize any such Cantor set as the weakly pseudoconvex locus, while       Theorem \ref{main1} produces what may be regarded intuitively as   a ``Cantor forest'' as the weakly pseudoconvex locus; see Figure \ref{fig:cantor-forest}.

In higher dimensions, we have constructed a similar result as in Theorem \ref{main1}, although in this setting the construction produces pseudoconvex domains.

\begin{theorem}\label{main2}
Let $n\ge 3$ and $E\subset \mathbb R^{n-1}\subset \mathbb C^{n-1}$ be an arbitrary closed set. For every integer \(m\ge 2\) and every point $x_0\in E$, there exist    a smoothly  bounded  pseudoconvex domain \(\Omega\subset \mathbb C^n\) and a neighborhood $\mathcal U$ of the real line $\{  x_0+i0 \}\times (\mathbb R+i0) $ %$\{(x_0+i0, u)\in \mathbb C^n: u\in \mathbb R\} $ 
in $\mathbb C^{n}$,    such that the weakly pseudoconvex locus of $\Omega$ is precisely
\[
W=\{(z,w)\in b\Omega: \operatorname{Re} z\in E,\  \operatorname{Im} z=0 \}\cap\  \overline{ \mathcal U}.
\]
Moreover, the D'Angelo type  is less than or equal to $2m+2 $ at every point in $W$. %If, in addition, \(E\) is a smooth submanifold, then the D'Angelo type  is   equal to $2m $ at every point $p\in W$. ???
\end{theorem}

If   the prescribed closed set $E$ is of real dimension $n-2$,  then the construction can be altered to produce convex domains, as given below.

\begin{theorem}\label{main3}
Let $n\ge 3$ and $E\subset \mathbb R^{n-2}\subset \mathbb C^{n-2}$ be an arbitrary closed set. For every integer \(m\ge 2\) and every point $\hat x_0\in E$, there exist  a smoothly  bounded   convex domain \(\Omega\subset \mathbb C^n\)  and a neighborhood $\mathcal U$ of the real line $\{  (\hat x_0, 0)+i0 \}\times (\mathbb R+i0) $ % $\{((\hat x_0, 0)+i0, u)\in \mathbb C^n: u\in \mathbb R\} $ 
in $\mathbb C^{n}$,    such that the weakly pseudoconvex locus of $\Omega$ is precisely
\[
W=\{(z,w)\in b\Omega: \operatorname{Re} z\in  E\times \{0\},\  \operatorname{Im} z=0\} \cap\  \overline{ \mathcal U}.
\]
Moreover, the D'Angelo type  is less than or equal to $2m $ at every point in $W$.
\end{theorem}

%We provide an algorithm to construct bounded smooth pseudoconvex domains with bizarre weak loci. 
Our construction is based on three main ingredients.  First, 
 we make use of the Whitney extension theorem to construct hypersurface models whose weak loci agree  with the prescribed closed sets.  When \(n=2\), the construction can be carried out using elementary subharmonic
potentials, since the weak pseudoconvexity condition corresponds to  a linear Poisson equation in this
dimension. The    hypersurfaces   obtained by this construction are, however, initially
unbounded.  When $n\ge 3$, the weak pseudoconvexity condition is reduced to a fully nonlinear  partial differential equation. We instead use local solvability results for  the degenerate elliptic Monge--Amp\`ere equation,  due to Hong and Zuily \cite{HZ87} and Kallel-Jallouli \cite{Ka03}, to obtain  smooth convex/plurisubharmonic potentials and hence the  desired local hypersurface models. Secondly, in order to pass from   unbounded or local hypersurfaces  to  bounded global hypersurfaces  without changing the prescribed weak loci, we use a symmetric smooth regularized maximum function. This function is obtained by mollifying the absolute value function and is designed to preserve the monotonicity and positive semi-definiteness properties. Finally,  we introduce  an anisotropic ellipsoidal global barrier, which, together with the  regularized maximum function, allows us to perform the patching construction and yield smoothly bounded global domains,  while preserving the weakly pseudoconvex loci and keeping  the D'Angelo type under the  prescribed upper bound.

%\begin{remark}
%???The Euclidean subspace $\mathbb{R}^{n-1}$ can be generalized to any totally real $C^\omega$ submanifold $M \subset \mathbb{C}^{n-1}$ of real dimension $d \le n-1$, with $E \subset M$ serving as an arbitrary compact subset.
%\end{remark}

All   examples constructed in this paper  are of finite D'Angelo type.   Moreover, in each of these examples, the prescribed weak locus has zero surface measure on the boundary, and the corresponding Levi determinant vanishes to finite order at every point of this locus. These observations would   suggest the following natural questions for smooth pseudoconvex domain of finite   type. 
\medskip

 \noindent\textbf{Conjecture 1:}
   Let $\Omega$ be  a smoothly bounded pseudoconvex domain of finite D'Angelo type  in $\mathbb C^n, n\ge 3$. Then the collection of all weakly pseudoconvex points in $b\Omega$  has zero surface measure on \(b\Omega\).
\medskip

   \noindent\textbf{Conjecture 2:}
 Let $\Omega$ be  a smoothly bounded pseudoconvex domain of finite D'Angelo type  in $\mathbb C^n, n\ge 3$. Then the Levi determinant, regarded as a smooth function on
$b\Omega$, is not flat at any point of \(b\Omega\). Equivalently, the Levi determinant  does not vanish to infinite order along the  tangential directions of $b\Omega$ at any weakly pseudoconvex boundary point. %has finite order vanishing  along tangential directions    at every point of \(b\Omega\).
 \medskip

As shown in Proposition \ref{2to1},    Conjecture  2 implies Conjecture 1.  In complex dimension two, the conclusions of both conjectures hold, as a direct consequence of a result of Kohn \cite{Ko72}, together with the equivalence between finite D'Angelo type and finite commutator type; see also \cite{DA93} or Proposition \ref{cn=2}. In higher dimensions, to the best of our knowledge, both questions remain open.

The rest of the  paper is organized as follows. In Section \ref{pre}, we review the definitions of weakly pseudoconvex  locus and D'Angelo type, together  with several  examples. In  Section \ref{wh}, we  prove  the   Whitney extension theorem needed for the construction. Local models with the prescribed weakly pseudoconvex  loci are constructed in Section \ref{2t} in dimension two, and in Sections \ref{np} and \ref{nc} in higher dimensions. To patch the  local models with exterior barrier functions, we discuss a regularized maximum function in Section \ref{rm}. Finally, in Sections \ref{b2}--\ref{bn}, we construct global smoothly bounded domains and verify  the desired properties, thereby completing the proofs of  our main theorems.

\section{Preliminaries}\label{pre}

Let $\Omega$ be a smoothly  bounded pseudoconvex domain in $\mathbb C^n$ with a smooth   defining function $\rho$, $n\ge 2$. Namely, $\rho$ is a smooth real function near a neighborhood $U$ of the boundary  $b\Omega$ with $\Omega \cap U =\{ p\in U: \rho(p)<0\}$, $ b\Omega \cap U=\{ p\in U: \rho(p)=0\}$ and   $\nabla \rho\ne 0$ on $b\Omega$. 

\subsection{Levi determinant and weakly pseudoconvex locus}
\begin{definition} At a boundary point $p \in b\Omega$, the complex normal gradient in terms of coordinates $(z_1, \ldots, z_n)$ is given by the   vector 
\[
 {\partial}\rho(p) = \left.\left( \frac{\partial \rho}{\partial  {z}_1}, \frac{\partial \rho}{\partial {z}_2}, \dots, \frac{\partial \rho}{\partial  {z}_n} \right)\right|_p.
\]
The \textbf{complex tangent space} $T_p^{1,0}(b\Omega)$
is defined by
\begin{equation*}
T_p^{1,0}(b\Omega)
=
\left\{
\xi=\sum_{j=1}^n \xi_j \frac{\partial}{\partial z_j}
:
  \partial\rho(p)(\xi)=0
\right\} = \left\{
\xi=\sum_{j=1}^n \xi_j \frac{\partial}{\partial z_j}
:
\sum_{j=1}^n
\frac{\partial \rho}{\partial z_j}(p)\xi_j
=
0
\right\}.
\end{equation*}
\end{definition}

After identifying
$\xi=\sum_{j=1}^n \xi_j \frac{\partial}{\partial z_j}$
with its coefficient vector $(\xi_1,\ldots,\xi_n)\in\mathbb C^n$,
the condition for the complex tangent space $T_p^{1,0}(b\Omega) $  can also be written as
\begin{equation*}
\langle \bar\partial\rho(p),\xi\rangle_{\mathbb C^n}=0,
\end{equation*}
where $\langle\cdot,\cdot\rangle_{\mathbb C^n}$ denotes the
standard Hermitian inner product, taken to be linear in the second
variable.  

The intrinsic geometric behavior of the boundary is captured by the Levi form defined on the complex tangent space as  below. 
\begin{definition} 
Let $H_\rho = \left( \frac{\partial^2 \rho}{\partial z_j \partial \bar{z}_k} \right)$ be the full $n \times n$ complex Hessian matrix of $\rho$. The \textbf{Levi form} at a point $p\in b\Omega$ evaluated on a vector ${\xi} \in T_p^{1,0}(b\Omega)$ is the quadratic form
\begin{equation*}
L_\rho(p; {\xi}, \bar{{\xi}}) = {\xi}^* H_\rho(p) {\xi}
\end{equation*}
restricted to the $(n-1)$-dimensional subspace $T_p^{1,0}(b\Omega)$.
\end{definition}

Since the Levi form is intrinsically defined only on the complex tangent space, we shall use the complex Hessian restricted to this subspace. For \(p\in b\Omega\), choose an \(n\times (n-1)\) matrix \(E_p\) whose columns form an orthonormal basis of the complex tangent space \(T_p^{1,0}(b\Omega)\) with respect to the standard Hermitian metric on \(\mathbb C^n\). We define  the \textbf{restricted complex Hessian} at $p$ by 
\begin{equation*}
\tilde{H}_\rho(p) = E_p^* H_\rho(p) E_p.
\end{equation*}
This  is an $(n-1) \times (n-1)$ Hermitian matrix representing the Levi form of  \(b\Omega\) at \(p\) with respect to the chosen orthonormal basis of \(T_p^{1,0}(b\Omega)\). 

Although the matrix \(\widetilde H_\rho(p)\) depends on the choice of orthonormal basis, a different choice changes it only by unitary conjugation. Thus its eigenvalues are independent of this choice. We denote these restricted eigenvalues by \begin{equation*} \lambda_1(p)\le \lambda_2(p)\le \cdots \le \lambda_{n-1}(p). \end{equation*} The product \begin{equation*} \prod_{j=1}^{n-1}\lambda_j(p) = \det \widetilde H_\rho(p) \end{equation*} will be called the \textbf{Levi determinant} of \(b\Omega\) at \(p\), with respect to the defining function \(\rho\).

 \medskip

%The ordered eigenvalues of $\tilde{H}_\rho(z)$, denoted by $\lambda_1(z) \le \lambda_2(z) \le \dots \le \lambda_{n-1}(z)$, are precisely the restricted eigenvalues of $H_\rho(z)$. In particular, the product $\prod_{j=1}^{n-1} \lambda_j(z) $ is called the Levi determinant of $b\Omega$ at $z$.
The following simple example illustrates how to compute the restricted complex Hessian and its eigenvalue(s).  

\begin{example}
Consider the  pseudoconvex domain $\Omega = \{z \in \mathbb{C}^2 : |z_1|^2 + |z_2|^4 - 1 < 0\}$ with defining function $\rho(z) = z_1\bar{z}_1 + (z_2\bar{z}_2)^2 - 1$. We compute the restricted complex Hessian at the boundary point $p=(1,0)$.
The complex gradient is $\partial\rho = (\bar{z}_1, 2\bar z_2|z_2|^2)$, which at $p$ becomes $(1, 0)$. 
It follows that the complex tangent space $T_p^{1,0}(b\Omega)$ is spanned by $E_p = \begin{pmatrix} 0 \\ 1 \end{pmatrix}$. The full complex Hessian at $p$ is 
\[
H_\rho(p) =\left. \begin{pmatrix} 1 & 0 \\ 0 & 4|z_2|^2 \end{pmatrix} \right|_p   = \begin{pmatrix} 1 & 0 \\ 0 & 0 \end{pmatrix}.
\]
Therefore, the restricted Hessian matrix at $p$ is 
\[
\tilde{H}_\rho(p) = E_p^* H_\rho(p) E_p = \begin{pmatrix} 0 & 1 \end{pmatrix} \begin{pmatrix} 1 & 0 \\ 0 & 0 \end{pmatrix} \begin{pmatrix} 0 \\ 1 \end{pmatrix} = 0.
\]
Hence the unique restricted eigenvalue is $\lambda_1 = 0$,   identifying $p$ as a point where the Levi determinant vanishes.  
\end{example}

While the individual restricted eigenvalues $\lambda_j(p)$ are extrinsic and depend heavily on  the scaling of the defining function, Sylvester's law of inertia guarantees that the signature of the Levi form is a biholomorphic invariant. In particular, the weak locus, the degeneracy of the Levi form as defined below, is independent of the choice of basis.   

\begin{definition} 
The \textbf{weakly pseudoconvex locus} (or simply, the \textbf{weak locus}) of a smoothly bounded pseudoconvex domain $\Omega$, denoted by $W \subset b\Omega$, is the   set of all boundary points where the   Levi determinant vanishes. Explicitly, it is the zero locus of the minimal restricted eigenvalue 
\begin{equation*}
W := \{p \in b\Omega : \lambda_1(p) = 0\}.
\end{equation*}
\end{definition}

\subsection{Alternative definition by bordered Hessian}
Computing restricted eigenvalues via point-wise orthonormal frames $E_p$ is useful locally, but it is not well suited for global arguments. In general, the complex tangent bundle need  not admit a global smooth orthonormal frame. We therefore replace the frame-dependent restricted Hessian by   a globally defined   bordered Hessian matrix.

\begin{definition} 
The  \textbf{bordered Hessian matrix} $H^{\text{bd}}_\rho $ associated with $\rho$ is the $(n+1) \times (n+1)$ Hermitian matrix 
\begin{equation*}
H^{\text{bd}}_\rho  = \begin{pmatrix} 0 &  \bar{\partial}\rho  \\ (\partial\rho)^T & H_\rho  \end{pmatrix}\ \ \text{on}\ \ b\Omega.
\end{equation*}
We define the corresponding \textbf{bordered Levi determinant}   by 
\begin{equation*}
 \Lambda_\rho  = -\det H^{\text{bd}}_\rho \ \ \text{on}\ \ b\Omega.
\end{equation*}
\end{definition}

The following known lemma shows that the bordered Levi determinant  agrees with the Levi determinant up to multiplication by a smooth positive factor.  For the reader's convenience, we provide the proof in Appendix \ref{ap}. 

\begin{lemma} \label{ld1}
Let $\rho$ be a smooth defining function for $\Omega$. Then, for every $p\in b\Omega$, the bordered Levi determinant satisfies \begin{equation*}  \Lambda_\rho(p) = |\partial\rho(p)|^2 \prod_{j=1}^{n-1}\lambda_j(p), \end{equation*} where $\lambda_1(p)\le \cdots\le \lambda_{n-1}(p)$ are the eigenvalues of the restricted complex Hessian $\widetilde H_\rho(p)$.    
\end{lemma}

As an immediate consequence of Lemma \ref{ld1}, we have the following characterization of the weak locus in terms of the bordered Levi determinant.

\begin{remark}  Since $|\partial \rho| $ is nowhere zero on the boundary of the smooth domain,    the invariant $ \Lambda_\rho$ exhibits exactly the same vanishing property as  the corresponding Levi determinant in view of Lemma \ref{ld1}.  Consequently, 
  the weakly pseudoconvex locus is characterized  by 
\begin{equation*}
W = \{p \in b\Omega :  \Lambda_\rho(p) = 0\}.
\end{equation*}
\end{remark}

We now examine  the behavior of weak  loci for two explicit examples below.

\begin{example}
Consider the domain $\Omega = \{(z, w) \in \mathbb{C}^2 : |z|^2 + |w|^4 - 1 < 0\}$. The defining function $\rho = |z|^2 + |w|^4 - 1$ is decoupled. By a direct computation,  the bordered complex Hessian 
\[
H^{\text{bd}}_\rho(z,w)
=
\begin{pmatrix}
0 & z & 2|w|^2w\\
\overline z & 1 & 0\\
2|w|^2\overline w & 0 & 4|w|^2
\end{pmatrix}.
\]
Hence by the cofactor expansion
\[
 \Lambda_\rho = -\det H^{\text{bd}}_\rho
=
z\left(4|w|^2\overline z\right)
+
2|w|^2w\left(2|w|^2\overline w\right) 
=
4|w|^2\left(|z|^2+|w|^4\right).
\]
Restricted on the boundary,  we further  have
 \[
 \Lambda_\rho=4|w|^2\qquad \text{on}\ \ b\Omega.
\]
Therefore
\[
 \Lambda_\rho=0
\quad\Longleftrightarrow\quad
w=0.
\]
Combining this with the boundary equation gives 
\(|z| =1.
\)
Hence the weakly pseudoconvex locus is
\[
W
=
\{(z,0):|z|=1\} 
 \cong_{\text{diff}} S^1.
\]
\end{example}

In the second example below, the weakly pseudoconvex locus is diffeomorphic to   \(S^1\times S^1\). 
\begin{example}
 Consider the domain 
 $
\Omega=\{(z,w)\in\mathbb C^2: |w|^2+16\bigl(|z|^2-1\bigr)^4-\frac14 <0\}.
$
For the defining function $\rho = |w|^2+16\bigl(|z|^2-1\bigr)^4-\frac14$, a straight forward computation gives  the   bordered complex Hessian  
$$
H^{\text{bd}}_\rho(z,w)
=
\begin{pmatrix}
0 & 64({|z|^2}-1)^3  z & w\\
64({|z|^2}-1)^3\overline z & 64({|z|^2}-1)^2(4{|z|^2}-1) & 0\\
 \overline w & 0 & 1
\end{pmatrix}.
$$
 Hence 
\begin{equation}\label{lae}
            \Lambda_\rho (z, w) =  \ 64^2\,{|z|^2}({|z|^2}-1)^6
+
64({|z|^2}-1)^2(4{|z|^2}-1)|w|^2.
     \end{equation}

Since $
16({|z|^2}-1)^4 = \frac14 - |w|^2\le \frac14 
$ on $b\Omega$, we have
 $
|{|z|^2}-1|\le \frac{1}{2\sqrt2}, 
$ and in particular,
$
4{|z|^2}>1
$
on \(b\Omega\).
Therefore   both terms in \eqref{lae} are nonnegative on $b\Omega$.  Consequently, if $ \Lambda_\rho =0$, then the first term  forces
$|z| =1$. The second term vanishes there as well. 
Substituting this into the boundary equation gives 
$
|w|=\frac12.
$
Thus the weakly pseudoconvex  locus is 
$$
  W =
\left\{(z,w)\in\mathbb C^2: |z|=1,\ |w|=\frac12\right\} 
\cong_{\text{diff}} S^1\times S^1.
 $$
% Because $\Omega$ is defined by polynomials, it contains no non-trivial complex varieties in its boundary by the Diederich-Fornaess theorem, confirming it is of finite D'Angelo type everywhere.
\end{example}

\subsection{Rigid domains}
A   domain $\Omega\subset \mathbb C^n$ is called rigid if there exists  a real-valued  function $f:\mathbb C^{n-1}\rightarrow \mathbb R$ such that 
\begin{equation}\label{rn}
     \Omega  = \{(z,w)\in \mathbb C^{n-1}\times \mathbb C: \rho =  \operatorname{Re} w + f(z, \bar z)<0 \}.
\end{equation}
Since $\frac{\partial \rho}{\partial w} = \frac{1}{2}$ does not vanish, $\rho$ is a defining function with nonvanishing gradient on $b\Omega$. Hence $\Omega$ has smooth boundary whenever   $f$ is smooth.  Due to the independence  along the $\text{Im} w$ direction, such domains are never bounded.

We next compute the complex tangent space, the Levi form, and the weakly pseudoconvex locus of a  rigid domain. They will be used repeatedly in our later construction.

\begin{lemma}\label{ri}
   Let $\Omega $ be a smooth rigid domain in $\mathbb C^n$  defined in \eqref{rn}. The following holds.
\begin{enumerate}
  \item The complex tangent space at $p\in b\Omega$ is
 \begin{equation*}
T^{1,0}_{p}(b\Omega)
=\text{span}_{\mathbb C}\left\{ L_j
=
\frac{\partial}{\partial z_j}
-
2f_{z_j}(p)\frac{\partial}{\partial w}: 
 j=1,\ldots,n-1\right\}. 
\end{equation*}
 \item With respect to the  complex tangential frame
\(\{L_1,\ldots,L_{n-1}\}\) as above, the restricted complex Hessian of \(\rho\) is exactly
\begin{equation*}
\tilde H_\rho= H_f=\left(f_{z_j\bar z_k}\right)_{j,k=1}^{n-1}.
\end{equation*}
 In particular, $\Omega$ is pseudoconvex if and only if $f$ is plurisubharmonic.

    \item If $f$ is plurisubharmonic, then the weakly pseudoconvex locus is
\begin{equation*}
    \begin{split}
        W &= \{p\in b\Omega: \det H_f(p) =0\}\\
        &= \{(z, -f(z, \bar z) +i v)\in \mathbb C^{n}: \det H_f(z) =0, v\in \mathbb R  \}.
    \end{split}
\end{equation*} 
%In particular, if $W$ is nonempty, then the Reeb direction $\operatorname{Im } w $ is always contained in $W$.
\end{enumerate}

\end{lemma}

\begin{proof}
Since the complex normal gradient is 
\begin{equation*}
\partial \rho    
= \left(f_{z_1}, \dots, f_{z_{n-1}}, \frac12\right), 
\end{equation*}
a \((1,0)\)-vector
 $ 
L
=
\sum_{j=1}^{n-1}{\xi}_j\frac{\partial}{\partial z_j}
+
{\xi}_n\frac{\partial}{\partial w}$  
is complex tangential to \(b\Omega\) if and only if
\begin{equation*}
\sum_{j=1}^{n-1}f_{z_j}{\xi}_j+\frac12{\xi}_n=0.
\end{equation*}
In particular, for each $j=1, \ldots, n-1$,  setting   $L_j
=
\frac{\partial}{\partial z_j}
-
2f_{z_j} \frac{\partial}{\partial w}$, the above equation is satisfied. Since $\{L_j, j=1, \ldots, n-1\}$ are linearly independent, they form a basis of the complex tangent space \(T_p^{1,0}(b\Omega)\).

The complex Hessian of \(\rho\) has the block form
\begin{equation*}
 H_\rho
=
\begin{pmatrix}
H_f & 0\\
0 & 0
\end{pmatrix},
\qquad
H_f=\left(f_{z_j\bar z_k}\right)_{j,k=1}^{n-1}.
\end{equation*}
Let \(  E\) be the \(n\times(n-1)\) matrix whose \(j\)-th column is
the coefficient vector of \(L_j \), $j=1,\ldots,n-1$.
Then
\begin{equation*}
  E
=
\begin{pmatrix}
I_{n-1}\\
-2f_z
\end{pmatrix},
\qquad
f_z=
\begin{pmatrix}
f_{z_1} & \cdots & f_{z_{n-1}}
\end{pmatrix}.
\end{equation*}
The restricted complex Hessian in the frame
\(\{L_1,\ldots,L_{n-1}\}\) is therefore
\begin{align*}
 \tilde H_\rho =  E^* H_\rho  E
&=
\begin{pmatrix}
I_{n-1} & -2f_z^*
\end{pmatrix}
\begin{pmatrix}
H_f & 0\\
0 & 0
\end{pmatrix}
\begin{pmatrix}
I_{n-1}\\
-2f_z
\end{pmatrix}  = H_f.
\end{align*}
 Thus \(b\Omega\) is pseudoconvex precisely when \(f\) is plurisubharmonic.  Moreover,  the weakly pseudoconvex locus is given by
\begin{equation*}
W
=
\{(z,w)\in b\Omega:\det H_f(z)=0\},
\end{equation*}
provided that \(H_f\ge 0\).
\end{proof}

Alternatively, one can directly compute  the corresponding  bordered complex Hessian 
\begin{equation*}
\mathcal H_\rho^{\mathrm{bd}}
=
%\begin{pmatrix}
%0 & f_{\bar z_1} & \cdots & f_{\bar z_{n-1}} &  \frac12\\
%f_{z_1} & f_{z_1\bar z_1} & \cdots & f_{z_1\bar z_{n-1}} & 0\\
%\vdots & \vdots & \ddots & \vdots & \vdots\\
%f_{z_{n-1}} & f_{z_{n-1}\bar z_1} & \cdots
%& f_{z_{n-1}\bar z_{n-1}} & 0\\
% \frac12 & 0 & \cdots & 0 & 0
%\end{pmatrix}  = 
\begin{pmatrix}
    0 & f_{\bar z}&\frac{1}{2}\\
    f^T_{z} &H_f& 0\\
    \frac{1}{2} &0&0
\end{pmatrix}.
\end{equation*}
The cofactor expansion gives the Levi determinant
\begin{equation*}
\Lambda_\rho = -\det \mathcal H_\rho^{\mathrm{bd}}
=
\frac14\det H_f.
\end{equation*}
from which the form of   the weakly pseudoconvex locus   follows.

\subsection{D'Angelo type}
The D'Angelo type for smooth hypersurfaces  was  introduced by D'Angelo to measure  the maximum order of contact of complex analytic varieties with the hypersurfaces; see, for example,  \cite{DA82, DA93}. We first recall the notion of  vanishing order. Let $g=(g_1,\ldots,g_N)$ be a smooth vector-valued function defined near $0\in\mathbb C$,  with $g(0)=0$.  We define \begin{equation*} \nu(g)=\min_{1\le j\le N}\nu (g_j) . \end{equation*} Equivalently, $\nu(g)=m$ if 
\[ g(\zeta)=P_m(\zeta,\bar \zeta)+O(|\zeta|^{m+1}), \]
where $P_m$ is a nonzero homogeneous polynomial vector of degree $m$ in $\zeta$ and $\bar \zeta$.

\begin{definition}\label{det}
Let $M$ be a smooth hypersurface in $\mathbb C^n$ with  smooth defining function $\rho$.  The \textbf{D'Angelo type}  of $M$ at a point $p\in M$ is defined by  \begin{equation}\label{tau}
     \tau(p): =\sup_\gamma \frac{\nu(\rho\circ \gamma)}{\nu(\gamma -p)}, 
\end{equation}
    where the supremum is taken over all germs of nonconstant  holomorphic curves $\gamma: (\mathbb C, 0)\rightarrow (\mathbb C^n, p)$. The point $p$ is called a point of \textbf{finite type} if $\tau( p)<\infty$. Otherwise, it is called  is a point of infinite type. 
       
        The hypersurface \(M\) is said to be of finite D'Angelo type if
\[
\sup_{p\in M}\tau(p)<\infty.
\] 
A smoothly bounded domain  $\Omega$ in $\mathbb C^n$  is said to be of finite D'Angelo type if its boundary $b\Omega$ is of finite D'Angelo type. 
\end{definition}

For   a smoothly bounded pseudoconvex domain $\Omega$ and a boundary point $p$, the condition \(\tau(p)=2\) is equivalent to strict pseudoconvexity at $p$, where the Levi form is strictly positive. Hence the weakly pseudoconvex locus $W$  can be  equivalently  described in terms of the D'Angelo type as
\begin{equation}\label{aw}
W =\{p\in b\Omega: \tau(p)>2 \}.\end{equation}

If the boundary  $b\Omega$
contains a germ of a nonconstant
holomorphic curve, then the corresponding D'Angelo type is infinite everywhere on the curve. As a consequence of \eqref{aw},  the image  of the whole curve is also contained in the weakly pseudoconvex locus $W$, where the Levi determinant vanishes.    In the real-analytic category, this connection is especially rigid. D'Angelo proved in \cite{DA82} that, if \(b\Omega\) is real-analytic near \(p\), then \(\tau(p)=\infty\) if and only if \(b\Omega\) contains a germ of a nonconstant  holomorphic curve passing through \(p\).  
 This equivalence fails in general for smooth hypersurfaces, as shown by a simple example below. 

 \begin{example}
     Let $M: = \{(z, w)\in \mathbb C^2: \rho = \operatorname{Re} w + e^{-\frac{1}{|z|^2}} = 0\}$, where $e^{-\frac{1}{|z|^2}}$ is understood to be extended smoothly by $0$ at $z=0$. Since $\rho\circ \gamma_0  =e^{-\frac{1}{|\zeta|^2}}
 $ along  the complex line $\gamma_0 (\zeta) =(\zeta, 0)$ and $e^{-\frac{1}{|\zeta|^2}} $ vanishes to infinite order at $0$, it follows that $0$ is a point of infinite D'Angelo type. 
 
 On the other hand, $M$ contains  no nonconstant germ of a holomorphic curve through $0$. Indeed, if $\gamma(\zeta) = (z(\zeta), w(\zeta))$ is such a curve, then $z(0)=w(0)=0$ and 
     $$ \operatorname{Re} w(\zeta)  =- e^{-\frac{1}{|z(\zeta)|^2}} \ \ \text{for all}\ \ \zeta\ \text{near}\ 0.$$
For every positive integer $k$, taking $k$ derivatives with respect to $\zeta$  on both sides and then setting $\zeta=0$,       one immediately obtains that the holomorphic function  $$w(\zeta)\equiv 0.$$ Consequently,
\(
e^{-\frac{1}{|z(\zeta)|^2}}\equiv 0,
\)
and hence \[z(\zeta)\equiv 0.\] Therefore \(\gamma\) is constant, proving that no nonconstant holomorphic curve germ is contained in \(M\).
 \end{example}

 Thus,   a point of infinite D'Angelo type on a hypersurface does not  in general  imply the existence of a complex analytic subvariety contained in the hypersurface. Moreover, Kim--Thu \cite[Example 2]{KT15} and Nguyen--Chu \cite{NC} constructed    some smooth real hypersurfaces with points of infinite type at which the supremum in \eqref{tau} of Definition \ref{det} is not attained. Equivalently, such hypersurfaces admit no holomorphic curve having infinite order contact with the hypersurface  at a  point of infinite type. We also refer  to     \cite{FT18} for pseudoconvex examples of this phenomenon. On the other hand, Forn\ae ss, Lee and the last author \cite{FLZ14} proved the existence of formal complex analytic curves at points of infinite  type.  
 
Despite the simplicity of the  definition of D'Angelo type, determining whether a hypersurface is of  finite type seems technically difficult in practice, even in the  real-analytic setting   in $\mathbb C^2$, due to the involvement of the implicit function theorem. The following elementary-looking example illustrates just that.  %: even for a hypersurface defined explicitly by smooth bump-type functions over the intersection of two cylinders, it is not immediate how to determine whether the D'Angelo type is finite.

\medskip

\noindent\textbf{Question. } 
   Let $c_j\in \mathbb C$, $d_j\in \mathbb R$ and $a_j, r_j\in \mathbb R^+$, $j=1, 2$, be constants,  and let
   $$U_j: =\{(z, w)\in \mathbb C^2: |z-c_j|^2 + |\operatorname{Im} w-d_j|^2<r_j^2\}, \ \ j=1, 2$$
   be  two cylinders in $\mathbb C^2$. Set $U: =U_1\cap U_2$. Suppose that $U\ne \emptyset$  and consider the  hypersurface $M$  in $U$ defined  by 
   $$M =\{ (z, w)\in U: \operatorname{Re}w  =   \sum_{j=1}^2 a_j e^\frac{1}{|z-c_j|^2 + |\operatorname{Im} w  -d_j|^2-r_j^2} \}.  $$
   Is   $M$  of finite D'Angelo type at every point of $M$? 
 
\medskip
   
Note that the hypersurface $M$ in the above is   real-analytic. However, since $M$ is not a closed variety, the global result of Diederich--Forn\ae ss does not apply to establish finite type. We  attempted to compute the type directly, but unfortunately failed to find an effective method.  When only one cylinder is involved, however, the finite type of the hypersurface within the cylinder can be established as follows.
\begin{example}
      Let $c\in \mathbb C$, $d\in \mathbb R$, and $r\in \mathbb R^+$ be constants, and let  $$U:  =\{(z, w)\in \mathbb C^2: |z-c|^2 + |\operatorname{Im} w-d|^2<r^2\}.$$ Consider the  hypersurface $M$  defined by 
   $$M =\{(z, w)\in U: \operatorname{Re} w =   e^\frac{1}{|z-c|^2 + |\operatorname{Im} w -d|^2-r^2} \}. $$ Then $M$  of finite D'Angelo type at every point of $M$.
\end{example}

\begin{proof}
The hypersurface $M$ is real-analytic in $U$. Suppose, by contradiction, that $M$ is of infinite D'Angelo type at some point $p\in M$. By D'Angelo's theorem, there exists a nonconstant holomorphic curve
$$
    \gamma(\zeta)=(z(\zeta),w(\zeta))
$$
such that $\gamma(0)=p$ and $\gamma(\zeta)\in M$ for all $\zeta$ sufficiently close to $0$.

Write $u=\operatorname{Re}w$ and $v=\operatorname{Im}w$. Along the curve $\gamma$, the defining equation of $M$ gives
$$
    u(\zeta) =
    e^{ \frac{1}{|z(\zeta)-c|^2+|v(\zeta)-d|^2-r^2} }.
$$
Since $(z(\zeta),w(\zeta))\in U$, we have $0<u(\zeta)<1$. Taking logarithms and then reciprocals, we obtain
\begin{equation*} 
     \frac{1}{\ln u(\zeta)}
    =
    |z(\zeta)-c|^2+|v(\zeta)-d|^2-r^2.
\end{equation*}
 Since $z$ and $w$ are holomorphic, while $u$ and $v$ are harmonic conjugates, applying the Laplacian with respect to $\zeta$ in the above gives
$$
    F''(u)|\nabla u|^2
    =
    4|z'|^2+2|\nabla v|^2,
$$
where 
$$
    F(s):=\frac{1}{\ln s},\qquad 0<s<1.
$$
Making use of the equalities 
$ 
    |\nabla u|^2=|\nabla v|^2=|w'|^2,
$ 
we further obtain
\begin{equation}\label{hh1}
    g(u(\zeta))|w'(\zeta)|^2=4|z'(\zeta)|^2,
\end{equation}
where
\begin{equation}\label{g}
       g(s):=F''(s)-2
    =
    \frac{2+\ln s}{s^2(\ln s)^3}-2,\qquad 0<s<1.
\end{equation}

We first consider the case  $z'\equiv 0$. Then \eqref{hh1} implies
$$
    g(u(\zeta))|w'(\zeta)|^2\equiv 0.
$$
If $w'\equiv 0$, then $\gamma$ is constant, a contradiction. Otherwise, on an open set where $w'\neq 0$, we have
$$
    g(u(\zeta))\equiv 0.
$$
Since $u=\operatorname{Re} w$ and $w$ is holomorphic and nonconstant on this set,  $u$ is  nonconstant there as well. Therefore, the  image of $u$ contains a nontrivial interval. Hence $g$ must vanish identically on an interval. This is impossible from the explicit expression of $g$ in \eqref{g}.

It remains to consider the case $z'\not\equiv 0$. Then \eqref{hh1} also implies that $w'\not\equiv 0$. Hence there exists a point near which both $z'$ and $w'$ are nonzero. On a sufficiently small neighborhood, say $V$,  of such a point, we may take logarithms in \eqref{hh1}. This gives
$$
    \ln g(u(\zeta))+\ln |w'(\zeta)|^2
    =
    \ln 4+\ln |z'(\zeta)|^2.
$$
Since $z'$ and $w'$ are holomorphic and nonvanishing on this neighborhood, $\ln |z'|^2$ and $\ln |w'|^2$ are harmonic. Therefore,
$$
    \Delta(\ln g(u(\zeta)))=0.
$$
Noting that  $u$ is harmonic, this further yields
$$
    G''(u(\zeta))|\nabla u(\zeta)|^2=0,
$$
where  
$$
    G(s):=\ln g(s)
    =
    \ln\left(\frac{2+\ln s}{s^2(\ln s)^3}-2\right).
$$
But $w'\neq 0$ on the chosen neighborhood, so $|\nabla u|^2=|w'|^2\neq 0$. Hence
$$
    G''(u(\zeta))\equiv 0 \ \ \text{on}\ \ V.
$$
Since $u$ is nonconstant on $V$ by assumption, its image contains an interval. Thus $G''\equiv 0$ on an interval, and consequently $$G(s)=as+b$$ there for some constants $a,b$. This contradicts the explicit expression of $G$, and thus completes the proof. 
\end{proof}

\section{Whitney extension theorem}\label{wh}

The classical Whitney extension theorem may be stated as follows. Let $E$ be a closed subset of $\mathbb R^n$. Suppose that $\{f_\alpha\}$ is a collection of functions on $E$, indexed by all multi-indices $\alpha$, satisfying the following compatibility condition: for every $m\in\mathbb N$ and every multi-index $\alpha$ with $|\alpha|\le m$,
$$
    f_\alpha(x)
    =
    \sum_{|\beta|\le m-|\alpha|}
    \frac{f_{\alpha+\beta}(y)}{\beta!}(x-y)^\beta
    +
    R_\alpha(x,y) 
    \qquad\text{for all}\ \  x,y\in E,
$$
where the function 
$ 
    R_\alpha(x,y)=o\left(|x-y|^{m-|\alpha|}\right)
$ 
uniformly as $x\to y$ in $E$. Then there exists a smooth function $F$ on $\mathbb R^n$ such that
$$
    F=f_0 \quad \text{on } E,
    \qquad
    D^\alpha F=f_\alpha \quad \text{on } E,
$$
for every multi-index $\alpha$. %Moreover, $F$ may be chosen to be real-analytic at every point of $\mathbb R^n\setminus E$. 
See, for instance, \cite{Wh34, Ma66, Na85, Ho90}.

We shall need the following  variant of the  Whitney extension theorem for existence of   nonnegative  smooth functions that are flat on a given closed set $E$ in the later construction. Recall that a smooth function $f$ is said to be flat at a point   if all its derivatives vanish at that point.

\begin{proposition}\label{Wf}
    Let $E $ be a closed subset of $\mathbb R^n$. Then there exists a smooth function $\psi $ on $\mathbb R^n$ such that
    \begin{enumerate}
        \item $\psi\ge 0$ on $\mathbb R^n$;

        \item $\psi^{-1}(0)= E$;

        \item $\psi$ is flat at every point of $E$.

%        \item  $\psi$  is    real-analytic everywhere away from   $   E$.     
    \end{enumerate}
    \end{proposition}

\begin{proof}
   Since  $\mathbb R^n$ is paracompact (i.e.,  every open cover has a locally finite refinement) and $\mathbb R^n\setminus E$ is open, there exists a countable locally finite  collection of open balls $ \{B_{r_j}(c_j)\}_{ j\in \mathbf N}$ in $\mathbb R^n$ centered at $c_j$ with radius $r_j>0$, such that 
   $$\mathbb R^n\setminus E = \bigcup_{j=1}^\infty B_{r_j}(c_j).$$
   See Lemma \ref{covering}. Here, local finiteness means that  every point   has a neighborhood intersecting only finitely many balls in the cover.

   For each $j\in \mathbb N$, define
   $$ \psi_j(x) =   \begin{cases} e^\frac{1}{|x-c_j|^2-r_j^2},\ \ &x\in B_{r_j}(c_j); \\0,  \ \ &\text{otherwise.}         \end{cases}
     $$
    Then $\psi_j\in C^\infty(\mathbb R^n)$, $\psi_j>0$ on $B_{r_j}(c_j)  $ and is flat everywhere else. Choose constants $a_j> 0$  such that 
     \begin{equation}\label{ac}
         a_j \sup_{|x|<j} \sum_{\alpha\in \mathbb N^n, |\alpha|\le j}|D^\alpha \psi_j(x)|\le 2^{-j},
     \end{equation} 
and set  $$\psi: = \sum_{j=1}^\infty a_j \psi_j.$$
By the local finiteness of the covering, the function $\psi$ is  smooth on $C^\infty(\mathbb R^n\setminus E)$. 
   
 To see that $\psi$ is smooth across  $E$, we show that for each   $R>0$ and  multi-index $\lambda \in \mathbb{N}^n$,    
$$  \sum_{j=1}^\infty  \sup_{|x|<R}  |a_j D^\lambda \psi_j(x)|    $$
converges uniformly on $\{|x|<R\}$. Indeed, let $N \in \mathbb{N}$ be such that    $N \ge \max\{R,  |\lambda|\}$. For $q > m > N$ we have
\begin{equation*} \begin{split}
      \sum_{j=m+1}^q  \sup_{|x|<R} | a_j D^\lambda \psi_j(x) |          &\le \sum_{j=m+1}^q a_j \sup_{|x|<N} \sum_{\nu \in \mathbb{N}^n, |\nu|\le N} |D^\nu \psi_j(x)| \\ 
    &\le \sum_{j=m+1}^q a_j \sup_{|x|<j} \sum_{\nu \in \mathbb{N}^n,  |\nu|\le j} |D^\nu \psi_j(x)|  
     \le \sum_{j=m+1}^q 2^{-j},  \end{split}\end{equation*}
where   the last step follows from  \eqref{ac}.  Hence the derivative series converges uniformly on compact subsets of $\mathbb R^n$. Therefore $\psi\in C^\infty(\mathbb R^n)$ and \begin{equation}\label{dpsi}
     D^\lambda\psi = \sum_{j=1}^{\infty}a_jD^\lambda\psi_j 
\end{equation} for every multi-index $\lambda$.

We now verify the desired properties of $\psi$. Since each $\psi_j\ge 0$ and each $a_j>0$, we have $\psi\ge 0$ on $\mathbb R^n$. 
Moreover, $$ \psi(p)=0 \iff \psi_j(p)=0 \ \text{for all } j \iff p\notin B_{r_j}(c_j) \ \text{for all } j \iff p\in E. $$ Thus $\psi^{-1}(0)=E$. Finally, if $p\in E$, then $p\notin B_{r_j}(c_j)$ for every $j$, and hence each $\psi_j$ is flat at $p$.
The flatness of $\psi$ on  $p \in E$ then follows from this together with \eqref{dpsi}. 
\end{proof} 

From the proof, one also sees that the function $\psi$ constructed above is real-analytic at every point of $$ \mathbb R^n\setminus \bigcup_{j=1}^{\infty} bB_{r_j}(c_j), $$
and is not real-analytic elsewhere.

\section{Unbounded convex models in dimension two}\label{2t}
In this section, we shall construct a   convex domain in $\mathbb C^2$ whose weak locus forms a trivial \(\mathbb R\)-family over the prescribed
closed set \(E\subset \mathbb R\), and hence is diffeomorphic to \(E\times\mathbb R\). Note that  $E$   may not be  real-analytic. Thus   we ought to carry out  the construction entirely within the smooth   category. 

Recall that a domain $\Omega\subset\mathbb R^n$ is called {\bf convex} if, for any two points  in $\Omega$, the line segment joining them is contained in $\Omega$.  For a domain $\Omega\subset\mathbb C^n$, convexity is understood after identifying $\mathbb C^n$ with $\mathbb R^{2n}$. 
A $C^2$ function $\rho$ on a convex   domain $D\subset\mathbb R^n$ is called {\bf convex} if  its real Hessian $D^2\rho$ is positive semi-definite at every point of $D$. %Consequently, if $\rho$ is convex on $D$, then each sublevel set
%$$   \{x\in D:\rho(x)<c\} $$ is convex. 
In particular,   a domain   defined as a sublevel set of a convex function on a convex ambient domain  is convex.

 \begin{proposition}\label{r2}
Let \(E\subset \mathbb R\subset \mathbb C\) be a closed set. For any integer \(m\ge 2\), there exists a  smooth  convex domain \(\Omega\subset \mathbb C^2\) such that the following hold.
\begin{enumerate}
    \item The weakly pseudoconvex locus $W$ of $\Omega$ is exactly
\[
W=\{(z,w)\in b\Omega: \ \operatorname{Re} z\in E, \ \operatorname{Im} z =0\}\cong_{\text{diff}} E\times \mathbb R;
\]
\item The D'Angelo type satisfies
\[
\tau(p)= 2m \qquad
\text{for every } \ \ p\in W.
\]
\end{enumerate}
 \end{proposition}

 \begin{proof}
     
 We shall define a rigid domain $\Omega$ by    \begin{equation*} 
   \Omega =\{(z, w)\in \mathbb C^2:     \operatorname{Re} w + f(z, \bar z)<0\}.  
 \end{equation*} 
Here in terms of coordinates $z=x+iy$ and $w=u+iv$, the  potential function $f$ will be chosen in  the  decoupled form   \begin{equation}\label{sf}
f(z, \bar z) = y^{2m} + \Phi(x),
\end{equation} 
where $m\in \mathbb Z^+$, and   $\Phi: \mathbb R\rightarrow \mathbb R$ is a smooth function.  Note that such a domain is always smooth but unbounded. 
\medskip

\noindent\emph{Choice of $\Phi$.} For the  given closed set $E\subset \mathbb R$, we shall construct $\Phi$ such that $\Phi''(x) = 0$ exactly on $E$ and $\Phi''(x) > 0$ everywhere else.  In fact, by   the Whitney extension theorem Proposition \ref{Wf}, there exists a smooth, non-negative function $\psi $ on $\mathbb{R}$ such that 
\begin{equation}\label{ei}
    \psi(x) = 0 \iff x \in E,
\end{equation} and $\psi$ is  flat at each point of $E$.     %To facilitate the global bounding phase, we additionally require $\psi(x)$ to be constant outside a large compact neighborhood, ensuring controlled growth at infinity for the global gluing.??? 
Integrate $\psi$  twice and set
\begin{equation}\label{di}
\Phi(x) = \int_0^x \int_0^t \psi(s) \, ds \, dt.
\end{equation}
By the Fundamental Theorem of Calculus,  $$\Phi'' = \psi\ \ \text{on}\ \ \mathbb R.$$  Thus   $\Phi $ satisfies the desired properties. Moreover, since $$\Phi' =\int_0^x \psi(t) \, dt $$ and $\psi (x)\ge 0$, inspecting separately the cases $x<0$ and $x\ge 0$ gives 
\begin{equation}\label{pp}
    x\Phi'(x)\ge 0\ \ \text{on}\ \ \mathbb R.
\end{equation}
This property will be crucially used in Sections \ref{b2} and \ref{b2'}. 
\medskip

\noindent\emph{Convexity.}  %Recall that a function is  convex if its real Hessian matrix is positive semi-definite. 
In the real coordinate frame $(x, y, u, v) \in \mathbb{R}^4$,    the real Hessian matrix of the defining function $\rho(z,w):  =\operatorname{Re} w + f(z, \bar z) $ is  
 \begin{equation*}
 D^2\rho (x, y, u, v) = \begin{pmatrix} \psi (x) & 0 & 0 & 0 \\ 0 & 2m(2m-1)y^{2m-2} & 0 & 0 \\ 0 & 0 & 0 & 0 \\ 0 & 0 & 0 & 0 \end{pmatrix}, 
 \end{equation*}
 which is positive semi-definite for all $m \ge 2$.  Thus, $\Omega$  is a sublevel set of the convex function $\rho$ in $\mathbb R^4$, and hence is a convex domain. In particular,  $\Omega$ is pseudoconvex. 
\medskip

\noindent\emph{Weak locus.} According to Lemma  \ref{ri}, the weak locus of $ \Omega $ is
\begin{equation*}
W=   \{(z,w)\in b\Omega:  
       f_{z\bar z}(z)=0\}.
\end{equation*}
By a direct computation, one has  
 \begin{equation*}
f_{z\bar{z}} = \frac{1}{4}\Delta f = \frac{1}{4}\left[ 2m(2m-1)y^{2m-2} + \psi(x) \right].
\end{equation*}
  Since both terms in the bracket are nonnegative, the Levi form  vanishes if and only if  
\begin{equation*}
y = 0 \quad \text{and} \quad \psi(x) = 0.
\end{equation*}
Combining this with \eqref{ei}, one has the weak locus 
\begin{equation*}
    \begin{split}
        W= &\{(z, w)\in b\Omega: \operatorname{Re} z\in E, \operatorname{Im} z=0 \}\\
        =& \{(x, 0, -\Phi(x), v)\in \mathbb R^4: x\in E\} \cong_{\text{diff}} E\times \mathbb R. 
    \end{split}
\end{equation*}

 \medskip

\noindent\emph{D'Angelo type.} At points off the weak locus $W$, since $\Omega$ is strictly pseudoconvex, the D'Angelo type is equal to $2$. To calculate the type effectively at a fixed  point $$p = (z_p, w_p) = (x_p+i0, -\Phi(x_p)+iv_p)\in W$$   where $x_p \in E$, we need to normalize the coordinates system.
   %We perform a real Taylor expansion of $\phi(x)$ around the center $x_p$. 

Since $\Phi''(x) = \psi(x)$ and $\psi$ is flat at $x_p\in E$, the formal Taylor series of $\Phi$  around the center $x_p$ terminates exactly after the linear term, leaving a smooth remainder function $R(x)$ that is   flat at $x_p$:
\begin{equation*}
\Phi(x) = \Phi(x_p) + \Phi'(x_p)(x - x_p) + R(x).
\end{equation*}
 Substituting this expansion back into our  defining function yields 
\begin{equation*}
\rho(z, w) = \operatorname{Re}\,w + y^{2m} + \Phi(x_p) + \Phi'(x_p)(x - x_p) + R(x).
\end{equation*}
Noting that $x_p$, $\Phi(x_p)$ and $\Phi'(x_p)$ are  real,    we collect the linear and constant terms into a single pluriharmonic expression:
\begin{equation*}
\operatorname{Re}\,w + \Phi(x_p) + \Phi'(x_p)\operatorname{Re}(z - z_p) = \operatorname{Re}\left[ w + \Phi(x_p) + \Phi'(x_p)(z - z_p) \right].
\end{equation*}
We define a global biholomorphic coordinate transformation $(z, w) \to (\tilde{z}, \tilde{w})$ via
\begin{align*}
&\tilde{z} = z - z_p; \\
&\tilde{w}  = w + \Phi(x_p)  + \Phi'(x_p)(z - z_p).
\end{align*}
 In these normalized coordinates, the defining function simplifies   to 
\begin{equation*}
\tilde{\rho}(\tilde{z}, \tilde{w}) = \operatorname{Re}\,\tilde{w} + \tilde{y}^{2m} + \tilde{R}(\tilde{x}),
\end{equation*}
where $\tilde{R}(\tilde{x})$ is   flat at $\tilde{x} = 0$.  
%For this reason, we shall only need to investigate the  D'Angelo type of   $$\rho(z, w) = \operatorname{Re}\,w + y^{2m} + R(x) = 0 $$ at the origin, where the smooth function $R$ is flat at $0$. 

Let 
\begin{equation*} \gamma(\zeta)=({\tilde z}(\zeta),{\tilde w}(\zeta)) \end{equation*} be a nonconstant holomorphic curve with \(\gamma(0)=0\). Restricting $\tilde \rho$ on $\gamma$, we have
$$ ({\tilde \rho}  \circ \gamma)(\zeta) =\operatorname{Re} {\tilde w}(\zeta) +    |\operatorname{Im} {\tilde z}(\zeta) |^{2m} + {\tilde R}(\operatorname{Re}{\tilde z}(\zeta)).$$
Since \({\tilde R}\) is flat at \(0\), the pullback  ${\tilde R}(\operatorname{Re}{\tilde z}(\zeta)) $ vanishes to infinite order. Thus it does not affect the D'Angelo type.

If $\nu({\tilde w}(\zeta))< \nu({\tilde z}(\zeta))$, then \(\operatorname{Re}{\tilde w}(\zeta)\) gives the leading term of \({\tilde \rho}\circ\gamma\), and hence \begin{equation*} \frac{\nu({\tilde \rho}\circ\gamma)} {\nu(\gamma)} = 1. \end{equation*}  
We next consider  $\nu({\tilde w}(\zeta))\ge  \nu({\tilde z}(\zeta))$. 
The leading nonharmonic term in \(  (\operatorname{Im}{\tilde z}(\zeta))^{2m} \) has order \(2m\nu({\tilde z}(\zeta))\). The term \(\operatorname{Re}{\tilde w}(\zeta)\) is harmonic, and therefore can  not cancel the leading nonharmonic term  in   $(\operatorname{Im}{\tilde z}(\zeta))^{2m} $.  
Hence      $   \nu({\tilde \rho}\circ\gamma)\le 2m\nu({\tilde z}(\zeta)),   $ and \begin{equation*} \frac{\nu({\tilde \rho}\circ\gamma)} {\nu(\gamma)}\le \frac{ 2m\nu({\tilde z}(\zeta))}{  \nu({\tilde z}(\zeta))} =2m. \end{equation*}  
  On the other hand, taking $\gamma(\zeta)=(\zeta,0),$  we obtain \begin{equation*} ({\tilde \rho}\circ\gamma) (\zeta) = (\operatorname{Im}\zeta)^{2m} + {\tilde R}(\operatorname{Re}\zeta). \end{equation*} Altogether, we have shown that the D'Angelo type at $0$ is \begin{equation*} \tau(0)=2m.
\end{equation*}
The proof is complete. 
 \end{proof}

There are two remarks in order concerning the proof of Proposition \ref{r2}.
\begin{enumerate}
\item As shown in the proof of Proposition \ref{r2},   the flatness of $\psi$ at points in $E$ makes it necessary to add the extra convex term $|y|^{2m}$ to the potential $f$   in \eqref{sf}. This term ensures that the boundary has finite D'Angelo type. Without it, the complex tangent line would have infinite  type at every point lying over $E$. 

    \item 
Alternatively, since the domain $\Omega$ is convex,   the  D'Angelo   type can be checked by the corresponding order of contact by  complex lines according to a theorem of Boas-Straube \cite{BS}. We leave the details to the interested reader.

\end{enumerate}

\section{Local pseudoconvex models in higher dimensions}\label{np}

In this section, we construct local pseudoconvex domains in $\mathbb C^n$ with prescribed weak loci. 
The construction, relies on a   local existence theorem of  Kallel-Jallouli \cite{Ka03} for smooth plurisubharmonic solutions to the complex Monge--Amp\`ere equation, which we use to obtain the required potential function, together  with a technically involved verification of the D'Angelo type. 

\begin{theorem}\cite{Ka03}\label{Ka}
  Let $g$ be a smooth nonnegative function defined in a neighborhood of $z_0\in \mathbb C^n$. Suppose that  $g$ is not flat at $z_0$. Then there exist   a neighborhood $U$ of $z_0$ and  a real-valued smooth plurisubharmonic function $f$ on $U$ such that 
    $$\det H_f(z, \bar z) = g(z, \bar z)\ \ \text{on}\ \ U. $$
\end{theorem}

  \begin{proposition}\label{lp}
Let \(E\subset \mathbb R^{n-1}\subset \mathbb C^{n-1}\) be a closed set. For any integer $m\ge 2$ and every point $x_0\in E$, there exist  a neighborhood $U$ of \(x_0+i0\) in $\mathbb C^{n-1}$ and a hypersurface $M$ defined in $U\times \mathbb C\subset \mathbb C^n$ such that the following holds.
\begin{enumerate}
    \item $M$ bounds a pseudoconvex side.
    \item The weakly pseudoconvex locus $W$ of $M$ is precisely 
    \[
W=\{(z,w)\in M: \operatorname{Re} z\in E, \ \operatorname{Im} z=0\}.
\]
\item \(M\) is of finite D'Angelo type, with type less than or equal to $2m+2$.
\end{enumerate}

%a smooth hypersurface and let \(\psi\in C^\infty(\mathbb R^{n-1})\) satisfy \(\psi\ge 0\), \(\psi^{-1}(0)=E\), and suppose that \(\psi\) is flat at every point of \(E\). Write \(z=x+iy\in \mathbb C^{n-1}\). Suppose that \(f \) is a smooth plurisubharmonic function in a neighborhood $U$ of \(z_0=(x_0,0)\in E\times\{0\}\)  satisfying
%\[\det H_f(z)=|y|^{2m}+\psi(x).\]
%Then the rigid hypersurface defined by 
%\[M=\{(z,w)\in U\times\mathbb C: \operatorname{Re}w+f(z) =0\}.\]
%bounds a  pseudoconvex side. 
\end{proposition}

\begin{proof}
Since $E$ is a closed set in $\mathbb R^{n-1}$,    the Whitney extension theorem Proposition \ref{Wf} yields a function  \(\psi\in C^\infty(\mathbb R^{n-1})\) such that   $$\psi\ge 0, \qquad \psi^{-1}(0)=E,$$
and  \(\psi\) is flat at every point of \(E\). Write \(z=x+iy\), with \(x,y\in\mathbb R^{n-1}\), and consider the  nonnegative smooth function 
$$ |y|^{2m}+\psi(x). $$ 
Given $x_0\in E$, since  $y^{2m}$ has finite order vanishing at $y=0$, the function $|y|^{2m}+\psi(x)$ is not flat at $z_0=x_0+i0\in \mathbb C^{n-1}$. Therefore, one can  apply Theorem \ref{Ka} by Kallel-Jallouli to obtain a smooth plurisubharmonic function $f$ on a neighborhood  $U$  of $z_0$ in $\mathbb C^{n-1}$  such that 
\begin{equation}\label{pnf}
    \det H_f(z, \bar z) = |y|^{2m}+\psi(x)  \ \ \text{on}\ \ U.
\end{equation}
In particular, the complex Hessian \(H_f\ge 0\). The additional term $|y|^{2m}$  is essential here: it not only  allows one to apply  the existence theorem of Kallel-Jallouli, but also provides the    finite type control needed along the weakly pseudoconvex locus,  as will be seen later.

Define the hypersurface by 
\[
M=\{(z,w)\in U\times\mathbb C: \rho = \operatorname{Re}w+f(z, \bar z) =0\}.
\]
 Since the defining function  is plurisubharmonic, $M$ bounds a  pseudoconvex side. 
\medskip

\noindent\emph{Weak locus.}
Both terms on the right-hand side of \eqref{pnf} are nonnegative. Thus for all $z\in U$, $$\det H_f(z, \bar z)=0 \qquad \iff\qquad y=0,\    x\in E.$$ Hence by Lemma  \ref{ri},  the stated description of \(W\) holds.

\medskip

\noindent\emph{D'Angelo type.}  Fix a point $p =(z_p, w_p)\in W$, where $z_p= x_p+i0$ with $x_p\in E $, and $w_p = -f(z_p, \bar z_p) +iv_p $.  Replacing $w$ by $w+w_p$ which also subtracts the constant $f(z_p,\bar z_p)$ from $f$, we may assume that $$ p=(z_p,0) =(x_p+i0, 0) \qquad \text{and} \qquad f(z_p,\bar z_p)=0. $$ 
Let \(\gamma(\zeta)=(z(\zeta),w(\zeta))\) be a nonconstant holomorphic curve with \(\gamma(0)=p \). Set
\[
\delta=\nu(z(\zeta)-z_p),
\qquad
\alpha=\nu(w(\zeta)).
\]
Then \begin{equation}\label{nug}
    \nu(\gamma)=\min\{\delta,\alpha\}.
\end{equation} 
The pullback of $\rho$ by $\gamma$ is \[
\rho(\gamma(\zeta))=\operatorname{Re}w(\zeta)+f(z(\zeta), \overline{z(\zeta)}).
\] 
If the order $\delta =\infty$, then $z\equiv z_p$ and  $\nu(\rho\circ\gamma) = \nu(\operatorname{Re} w(\zeta)) = \alpha$. Hence
$$\frac{\nu(\rho\circ\gamma)}{\nu(\gamma)} =1. $$
In what follows, we assume that $\delta<\infty$.

Decompose \( f(z(\zeta), \overline{z(\zeta)})\) into a sum of a harmonic term,   and a non-harmonic term. Denote by \(L \)  the vanishing order of the nonharmonic term of $f(z(\zeta), \overline{z(\zeta)})$. Indeed, for \(N\ge 1\), write the Taylor polynomial of \(f(z(\zeta), \overline{z(\zeta)})\) at \(0\), of total degree at most \(N\), as
\(
P_N({\zeta},\overline {\zeta})=\sum_{p+q\le N}a_{pq}{\zeta}^p\overline {\zeta}^q.
\)
Let \(Q_N\) denote the mixed nonharmonic part of \(P_N\):
\[
Q_N({\zeta},\overline {\zeta})
=
\sum_{\substack{p,q\ge 1\\ p+q\le N}}
a_{pq}{\zeta}^p\overline {\zeta}^q.
\]
Define
\[
L_N=
\begin{cases}
\min\{p+q:p,q\ge 1,\ a_{pq}\ne 0\}, & Q_N\not\equiv 0;\\
N+1, & Q_N\equiv 0.
\end{cases}
\]
Then   
\[L =\lim_{N\to\infty}L_N\in\mathbb N\cup\{\infty\}.\]   

We claim that \(L \) is finite. Indeed, since \(x_p\in E\) and \(\psi\) is flat at \(x_p\), the term \(\psi(x(\zeta))\) vanishes to infinite order at \(0\). By the definition of \(\delta\), at least one component of \(z(\zeta)-z_p\) has order \(\delta\). Since \(z(\zeta)\) is holomorphic, the function \(|y(\zeta)|^2\) has order \(2\delta\). Therefore
\begin{equation}\label{nuh}
    \nu(\det H_f(z(\zeta), \overline{z(\zeta)}))=   \nu( |y(\zeta)|^{2m} +\psi(x(\zeta)))=\nu(|y(\zeta)|^{2m}) = 2m\delta.
\end{equation}
Suppose, toward a contradiction, that \(L=\infty\), meaning   that all non-harmonic terms of $F(\zeta): =f(z(\zeta), \overline{z(\zeta)})$ vanish to infinite order.  In particular,  \(F_{{\zeta}\overline {\zeta}}\) is flat at \(0\).  By the chain rule,
\[
F_{{\zeta}\overline {\zeta}}(\zeta)
=
\sum_{j,k=1}^{n-1}
f_{z_j\overline z_k}(z(\zeta))z_j'(\zeta)\overline{z_k'(\zeta)}
=
\left\langle H_f(z(\zeta), \overline{z(\zeta)})z'(\zeta),z'(\zeta)\right\rangle.
\]
Let \(\lambda_{\min}(\zeta)\) be the smallest eigenvalue of \(H_f(z(\zeta), \overline{z(\zeta)})\). Since \(H_f(z(\zeta), \overline{z(\zeta)})\ge 0\), for \(\zeta\ne 0\) we have
\[
0\le \lambda_{\min}(\zeta)
\le
\frac{\left\langle H_f(z(\zeta), \overline{z(\zeta)})z'(\zeta),z'(\zeta)\right\rangle}{\|z'(\zeta)\|^2}
=
\frac{F_{{\zeta}\overline {\zeta}}(\zeta)}{\|z'(\zeta)\|^2}.
\]
The numerator is flat, while \(\|z'(\zeta)\|^{2}\) is of vanishing order $2(\delta-1) $. Thus    $$\lambda_{\min}(\zeta)= O(|\zeta|^k) \qquad \text{for all}\ \ k\ge 0.$$ On the other hand,  the remaining eigenvalues of \(H_f(z(\zeta), \overline{z(\zeta)})\) are locally bounded. Indeed, since the matrix \(H_f\) is positive semi-definite  by pseudoconvexity,   its eigenvalues are nonnegative and bounded above by the trace, which is continuous and therefore locally bounded. Consequently, \(\det H_f(z(\zeta), \overline{z(\zeta)})\) is flat at \(0\).  This is a contradiction to \eqref{nuh}. Hence the claim is proved.

Fix an integer $N>L+2$. Decompose the Taylor
expansion of $f(z(\zeta), \overline{z(\zeta)})$ up to order $N$ as
\begin{equation*}
f(z(\zeta), \overline{z(\zeta)})=h(\zeta)+g(\zeta)+O(|\zeta|^{N+1}),
\end{equation*}
where the polynomials $h(\zeta)$ and  $g(\zeta)$ denote  the corresponding harmonic part and  the non-harmonic part, respectively. In particular, since $f(z_p, \bar z_p)=0$, 
 \begin{equation}\label{gh}
     \nu(g(\zeta))\ge \nu(z(\zeta)-z_p)= \delta   \qquad\text{and}\qquad \nu(h(\zeta))\ge \nu(z(\zeta)-z_p)= \delta. 
 \end{equation}
 %, namely the sum of the mixed terms $t^p\overline t^q$ with $p,q\ge 1$. 
Since $h(\zeta)$ is harmonic, there exists a holomorphic polynomial $H(\zeta)$ with $H(0)=0$ such that
\begin{equation*}
h(\zeta)=\operatorname{Re}H(\zeta),
\end{equation*}
%We shall use the following elementary observation. Let $P_N(t,\overline t)$ be a polynomial in $t$ and $\overline t$ of total degree at most $N$. Decompose it uniquely as
%\begin{equation*}
%P_N(t,\overline t)=H_N(t,\overline t)+G_N(t,\overline t),
%\end{equation*}
%where $H_N$ consists only of the pure monomials, namely those with
%$p=0$ or $q=0$, and $G_N$ consists only of the mixed monomials
%$t^p\overline t^q$ with $p,q\ge 1$. If
%\begin{equation*}
%\beta=\nu (H_N), \qquad L=\nu( G_N),
%\end{equation*}
%then
%\begin{equation*}
%\nu (P_N)=\min\{\beta,L\}.
%\end{equation*}
%Indeed, pure monomials and mixed monomials are distinct monomials in the two independent variables $t$ and $\overline t$. Hence they span complementary subspaces of the space of polynomials of degree at most $N$, and therefore a pure term cannot cancel a mixed term.
and thus 
\begin{equation*}
\rho(\gamma(\zeta))
=
\operatorname{Re} (w(\zeta)+H(\zeta)) +g(\zeta)+O(|\zeta|^{N+1}). 
\end{equation*}
Note that by definition of $L$ and \eqref{gh},
\begin{equation}\label{Ho}
L=\nu (g(\zeta))\ge \delta \qquad\text{and}\qquad     \nu(H(\zeta)) =\nu(h(\zeta))\ge \delta.
\end{equation}
Since %pure monomials and mixed monomials are distinct monomials in the two independent variables ${\zeta}$ and $\overline {\zeta}$,  they span complementary subspaces of the space of polynomials of degree at most $N$, and therefore 
a pure term cannot cancel a mixed term, 
we obtain
\begin{equation*} 
\nu (\rho\circ \gamma)=\min\{\beta,L\},
\end{equation*}
where
\begin{equation*}
\beta  =
\nu\bigl( \operatorname{Re} (w(\zeta)+H(\zeta))\bigr) =
\nu\bigl( w(\zeta)+H(\zeta)\bigr).
\end{equation*}
Consequently, together with \eqref{nug} one has
\begin{equation}\label{nur}
\frac{\nu (\rho\circ\gamma)}
{\nu(\gamma)}
=
\frac{\min\{\beta,L\}}{\min\{\alpha, \delta\}}. 
\end{equation}
 
%This formula is important: in general one cannot replace $\beta$ by$\alpha$, because the term $\operatorname{Re}w(t)$ may partially or completely cancel the harmonic part $h(t)$.

%Since $w(t)$ isholomorphic, $\operatorname{Re}w(t)$ has the same order as $w(t)$, unless $w\equiv 0$. Thus
 We now estimate \eqref{nur} by splitting into two cases.  First suppose that $\delta> \alpha$. Then \begin{equation*}
\nu(\gamma)=\min\{\alpha,\delta\}=\alpha.
\end{equation*}
 %We next compare $\beta$ with the order of the curve. %he harmonic part of $f$ is the real part of a holomorphic function vanishing at $z_p$. 
Since $ w(\zeta)$ has order $\alpha$, whereas $H(\zeta)$ has order at
least $\delta$ by \eqref{Ho},  no cancellation at order $\alpha$ is possible for $ w(\zeta) +H(\zeta) $. 
Hence
\begin{equation*}
\beta=\alpha.
\end{equation*}
Consequently,
\begin{equation*}
\frac{\nu (\rho\circ\gamma)}
{\nu(\gamma)}
=
\frac{\min\{\alpha,L\}}{\alpha}
\le 1.
\end{equation*}

Now suppose that $\alpha\ge \delta$. Then \eqref{nur} is reduced to 
\begin{equation}\label{de}
\frac{\nu(\rho\circ\gamma)}
{\nu(\gamma)}
=
\frac{\min\{\beta,L\}}{\delta}
\le
\frac{L}{\delta}.
\end{equation}
 %Since the harmonic part of \(f(z(t))\) may cancel with \(\operatorname{Re}w(t)\), while the non-harmonic part cannot, 
%\[
%\nu(\rho(\gamma(t)))\le \min\{\alpha,L\},???
%\]
%where  \(L \) denotes the vanishing order of the nonharmonic term of $f(z(t))$.
%Combining this with \eqref{nug}, we obtain
%\begin{equation}\label{de}
%    \frac{\nu(\rho\circ\gamma)}{\nu(\gamma)}
%\le
%\frac{\min\{\alpha,L\}}{\min\{\alpha,\delta\}}.
%\end{equation}
%Since the corresponding $  \frac{\nu(\rho\circ\gamma)}{\nu(\gamma)}$ equals $1$ if $z(t)\equiv z_0$, we shall assume below that $\delta<\infty$. 
To    estimate the vanishing order $L$, the nonharmonic part   of $ f(z(\zeta), \overline{z(\zeta)}) $, we   compare it  with that  of the diagonal entries $f_{z_j\bar z_j}(z(\zeta), \overline{z(\zeta)})$ of \(H_f(z(\zeta), \overline{z(\zeta)})\). We first claim that   for every \(j=1, \ldots, n-1\),
\begin{equation}\label{nud}
\nu (f_{z_j\overline z_j}(z(\zeta), \overline{z(\zeta)}))\ge L-2\delta.
\end{equation}
Indeed, write a typical mixed monomial in the Taylor expansion of $f$ at
$z_p$ in the form
\begin{equation*}
(z-z_p)^\alpha(\overline z-\overline z_p)^\beta,
\qquad |\alpha|\ge 1,\quad |\beta|\ge 1.
\end{equation*}
Such monomials are precisely the terms which contribute to the
non-harmonic part. If this monomial contains both
$z_j-z_{0,j}$ and $\overline z_j-\overline z_{0,j}$, so that
$\alpha_j\ge 1$ and $\beta_j\ge 1$, applying
$\partial_{z_j}\partial_{\overline z_j}$ lowers the two exponents
$\alpha_j$ and $\beta_j$ by one. Hence the total bidegree is lowered by two. After restricting to the curve $z(\zeta)$, the order in $\zeta$ is therefore
lowered by at most $2\delta$. 
On the other hand, if the monomial does not contain both
$z_j-z_{p,j}$ and $\overline z_j-\overline z_{p,j}$, then
$\partial_{z_j}\partial_{\overline z_j}$ annihilates that monomial. The 
first nonzero term which survives after applying
$\partial_{z_j}\partial_{\overline z_j}$ must therefore come from either
another lowest-order mixed monomial or from a mixed monomial of higher
order. In either case its order along the curve is no smaller than
$L-2\delta$. The claim is confirmed.

On the other hand,  for a positive semi-definite Hermitian matrix, Hadamard's inequality  bounds  the determinant  from above by the product of its diagonal entries. We apply this  to the   matrix \(H_f(z(\zeta), \overline{z(\zeta)})\) to get 
\[
0\le \det H_f(z(\zeta), \overline{z(\zeta)})
\le
\prod_{j=1}^{n-1} f_{z_j\overline z_j}(z(\zeta), \overline{z(\zeta)}),\qquad |{\zeta}|<<1.
\]
Taking vanishing orders of the above and making use of  \eqref{nuh} further give 
\[
\nu
\left(
\prod_{j=1}^{n-1} f_{z_j\overline z_j}(z(\zeta), \overline{z(\zeta)})
\right)
\le 2m\delta.
\]
Then     there exists \(k\in\{1,\ldots,n-1\}\) such that
\[
\nu( f_{z_k\overline z_k}(z(\zeta), \overline{z(\zeta)}))\le 2m\delta.
\]
 Combining \eqref{nud} with the index \(k\) chosen above gives
\[
L-2\delta
\le
\nu( f_{z_k\overline z_k}(z(\zeta), \overline{z(\zeta)}))
\le
2m\delta.
\]
Thus
\[
L\le (2m+2)\delta.
\]
Finally, invoking  \eqref{de} one infers  
%If \(\alpha\ge\delta\), this ratio is bounded by \(L/\delta\le 2m+2\). If \(\alpha<\delta\), it is bounded by \(1\). Thus every nonconstant holomorphic curve through \(p\) satisfies
\[
\frac{\nu(\rho\circ\gamma)}{\nu(\gamma)}
\le 2m+2.
\]
Taking the supremum over all such curves gives \(\tau(p)\le 2m+2\).
\end{proof}

\section{Local convex models in higher dimensions}\label{nc}

To construct convex domains with desired weak loci, we have to compromise by lowering one real dimension on the prescribed closed set. Our tool for constructing the potential is the following   special case of Hong--Zuily's result for the real degenerate  Monge--Amp\`ere equation. 

\begin{theorem}\label{HZ}\cite{HZ87} 
    Let $\psi$ be a smooth nonnegative function defined in a neighborhood of   $x_0\in \mathbb R^n$. Suppose that  $\psi$ is  not flat at $x_0$. Then there exist  a (convex) neighborhood $U$ of $x_0$ and  a smooth convex function to 
    $$\det D^2 \Phi  = \psi \ \  \text{on}\ \ U. $$
\end{theorem}

 \begin{proposition}\label{lc}
Let \(E\subset \mathbb R^{n-2}\subset \mathbb C^{n-1}\) be a closed set.  For every integer $m\ge 2$ and every point $\hat x_0\in E$, there exist  a neighborhood $U$ of $(\hat x_0, 0)$ in $\mathbb R^{n-1}$ and a hypersurface $M$ defined in $(U+i\mathbb R^{n-1})\times \mathbb C\subset \mathbb C^n$, such that the following holds.
\begin{enumerate}
    \item $M$ bounds a convex side.
    \item The weakly pseudoconvex locus $W$ of $M$ is precisely 
    \[
W=\{(z,w)\in M: \operatorname{Re} z\in E\times \{0\},\  \operatorname{Im} z=0\}.
\]
\item \(M\) is of finite D'Angelo type, with type less than or equal to $2m $. 
\end{enumerate}
  \end{proposition}

\begin{proof}
    As in the previous sections, we apply   the Whitney extension theorem Proposition \ref{Wf} to obtain      a smooth nonnegative function $\psi$ characterizing   the set $E$, such that \eqref{ei} holds
    and $\psi$ is flat at each point of $E$. For each point $\hat x_0\in E\subset  \mathbb R^{n-2}$, we  invoke  the local existence theorem of  Hong--Zuily, Theorem \ref{HZ},  to     obtain  a   smooth convex potential $\Phi $ and a convex neighborhood of $U$ of $ (\hat x_0, 0) $ in $\mathbb R^{n-1}$ such that 
\begin{equation}\label{phic}
\det D^2 \Phi(x) = |x_{n-1}|^{2m} + \psi(\hat x)\ \ \text{on}\ \ U,
\end{equation}
where  $x=(\hat x, x_{n-1})\in \mathbb R^{n-2}\times \mathbb R$. Set  
\begin{equation*}
f(z, \bar z):=|y|^{2m}+\Phi(x),
\end{equation*}
where   $y\in\mathbb R^{n-1}$ and $z=x+iy\in \mathbb C^{n-1}$. 
Then $f$ is convex on $U+i\mathbb R^{n-1}\subset \mathbb C^{n-1}$. Consequently, the rigid hypersurface   
\[
M=\{(z,w)\in (U+i\mathbb R^{n-1})\times\mathbb C: \operatorname{Re}w+f(z, \bar z) =0\}
\]
bounds the  convex side
$$ \Omega=\{(z,w)\in (U+i\mathbb R^{n-1})\times\mathbb C: \operatorname{Re}w+f(z,\bar z)<0\}. $$
%\medskip

\noindent\emph{Weak locus.}   Since \(f\) contains no mixed \(x,y\)-terms, its complex Hessian is given by 
\begin{equation*}
H_f(z, \bar z) =
\frac{1}{4}
\left(
D_x^2\Phi(x)+D_y^2|y|^{2m}
\right).
\end{equation*} In view of Lemma  \ref{ri}, it suffices to determine the zero set of \(\det H_f\). Equivalently, we shall prove that for $(x, y) \in U\times \mathbb R^{n-1}$,
\begin{equation}\label{eq}
\det \left(
D_x^2\Phi(x)+D_y^2|y|^{2m}
\right)=0
\qquad\Longleftrightarrow\qquad
y=0,\quad x_{n-1}=0,\quad \psi(\hat x)=0.
\end{equation}
The implication from right to left is straightforward when $m\ge 2$ from the construction of $f$. Indeed, if \(y=0\), then \(D_y^2|y|^{2m}=0\), and hence by \eqref{phic}
\begin{equation*} \det \left( D_x^2\Phi(x)+D_y^2|y|^{2m} \right) = \det D_x^2\Phi(x)=|x_{n-1}|^{2m}+\psi(\hat x), \end{equation*}
which vanishes when $x_{n-1}=0 $ and $ \psi(\hat x)=0 $.

Conversely,  suppose that the left-hand side of \eqref{eq} vanishes.
%Moreover,
%\begin{equation*}
%D_y^2|y|^{2m} =2m|y|^{2m-2}I + 4m(m-1)|y|^{2m-4}yy^T.
%\end{equation*}
% Hence
%\begin{equation*}
%\det H_F(z) = 4^{-n} \det \left(D^2\Phi(x)+2m|y|^{2m-2}I+4m(m-1)|y|^{2m-4}yy^T\right).
%\end{equation*}
Note that   \(D_x^2\Phi(x)\ge 0\) and 
\(D_y^2|y|^{2m}\) is positive definite except at \(y= 0\). 
 Thus, if $\det \left(
D_x^2\Phi(x)+D_y^2|y|^{2m}
\right)=0$, then $ y= 0,$ where   $ D_y^2|y|^{2m} =0 $. By \eqref{phic}
$$  
\det
\left(
D_x^2\Phi(x)\right)=
 |x_{n-1}|^{2m}+\psi(\hat x)=0.$$
But this equality can happen only when $x_{n-1} = \psi(\hat x)=0.$ We thus have the desired weak locus.

\medskip

\noindent\emph{D'Angelo type.}
%Again, we shall check the D'Angelo type by the complex line contact type in light of the  classical theorem of Boas-Straube \cite{BS}. 
To compute the type at a   point $p =\left( (\hat x_p, 0)+i0, w_p\right) \in W$, where $\hat x_p\in E$ and $w_p = -\Phi(\hat x_p, 0) +iv_p$,  we   perform a global biholomorphic coordinate shift $(z, w) \to ({\tilde z}, {\tilde w})$, which sends  $p$ to the origin and removes the linear components in the potential:
\begin{align*}
{\tilde z} &= z - (\hat x_p, 0);\\
{\tilde w} &= w - w_p + \sum_{j=1}^{n-1} \frac{\partial \Phi}{\partial x_j}(\hat x_p, 0) \tilde z_j.
\end{align*}
In these coordinates, the normalized local defining function takes the form  
\begin{equation*}
\tilde \rho ({\tilde z}, {\tilde w}) = \operatorname{Re}\,{\tilde w} + \tilde{\Phi}({\tilde x}) + |{\tilde y}|^{2m},
\end{equation*}
where  \begin{equation*} 
    \tilde{\Phi}({\tilde x}) = \Phi({\tilde x} + (\hat x_p, 0)) - \Phi(\hat x_p, 0) - \nabla\Phi(\hat x_p, 0) \cdot {\tilde x}. 
\end{equation*} 
Thus   $\tilde \Phi(0) = \tilde \Phi'(0) = 0$. 
Since $\Phi$ is  convex, it lies above its supporting hyperplane at $(\hat x_p,0)$. Hence  \begin{equation}\label{wp}
    \tilde{\Phi}({\tilde x})   \ge 0 
\end{equation} 
near the origin. 
 %It suffices to examine the complex line contact type at the origin. 

Let  $\gamma(\zeta)=({\tilde z}(\zeta),{\tilde w}(\zeta)) $ be a nonconstant holomorphic curve with \(\gamma(0)=0\).  Pulling back the normalized local defining function along this curve gives 
\begin{equation*}
({\tilde \rho}  \circ \gamma)(\zeta) =\operatorname{Re} {\tilde w}(\zeta) +  \tilde{\Phi}(\operatorname{Re} {\tilde z}(\zeta))  + |\operatorname{Im} {\tilde z}(\zeta) |^{2m}.
\end{equation*}
First consider  $\nu({\tilde w}(\zeta))< \nu({\tilde z}(\zeta))$. Since $\tilde \Phi$ vanishes to second order at the origin and $m\ge 2$, the function $\operatorname{Re} {\tilde w}(\zeta)   $ becomes the leading term of $({\tilde \rho}  \circ \gamma)(\zeta) $.  Hence $\nu({\tilde \rho}\circ r) = \nu({\tilde w}(\zeta))$, and 
$$ \frac{\nu({\tilde \rho}\circ\gamma)} {\nu(\gamma)} = 1. $$
We now  assume  $\nu({\tilde w}(\zeta))\ge  \nu({\tilde z}(\zeta))$. Let \(P_d\) be the first nonzero homogeneous term in the Taylor expansion of \(P(\zeta): = \tilde{\Phi}(\operatorname{Re}({\tilde z}(\zeta))) + |\operatorname{Im}({\tilde z}(\zeta))|^{2m}  \) at \(0\). Since $\tilde \Phi\ge 0$ by \eqref{wp}, the function $P $ is nonnegative. Moreover,  the term $|\operatorname{Im}\tilde z(\zeta)|^{2m}$   has vanishing order $2m\nu({\tilde z}(\zeta))$. Hence     $P_d$ is also  nonnegative, with degree  $$d\le 2m\nu({\tilde z}(\zeta)).$$  
Furthermore, since  $P_d(0)=0$, the nonnegative function $P_d$ cannot be harmonic by the maximum principle. On the other hand,   \(\operatorname{Re}{\tilde w}(\zeta)\) is harmonic. Thus    it cannot cancel    the nonharmonic polynomial $P_d$. Consequently, 
\begin{equation*} \nu\bigl(   {\tilde \rho}  \circ \gamma \bigr)\le d\le 2m\nu({\tilde z}(\zeta)), \end{equation*}
and 
\begin{equation*} \frac{\nu({\tilde \rho}\circ\gamma)} {\nu(\gamma)}\le \frac{ 2m\nu({\tilde z}(\zeta))}{  \nu({\tilde z}(\zeta))} =2m. \end{equation*}
This completes the   proof.
\end{proof}

\begin{comment}

In this coordinate framework, the complex tangent space is precisely $\{w^* = 0\}$. We parameterize an arbitrary complex line passing through the origin along a direction vector $v = (v_1, \dots, v_{n-1}) \in \mathbb{C}^{n-1} \setminus \{0\}$ using the complex parameter $\zeta = s + it \in \mathbb{C}$:
\begin{equation*}
\gamma(\zeta) = (\zeta v, 0).
\end{equation*}
Pulling back the normalized local defining function along this line gives:
\begin{equation*}
(\rho  \circ \gamma)(\zeta) = \tilde{\Phi}(\operatorname{Re}(\zeta v)) + |\operatorname{Im}(\zeta v)|^{2m}.
\end{equation*}
Because the normalized potential $\tilde \Phi$ satisfies \eqref{wp}, its Taylor series contribution can only add nonnegative subharmonic terms, and  thus can never cancel the strictly positive transverse penalty term $|\operatorname{Im}(\zeta v)|^{2m}$. 
Therefore, the lowest-degree non-vanishing term in the real Taylor expansion of the pullback is bounded exactly by the homogeneous degree $2m$ of  $|\operatorname{Im}(\zeta v)|^{2m}$. This guarantees that the order of contact satisfies 
\begin{equation*}
\tau(p) = \min\{\nu(\tilde{\Phi}(\operatorname{Re}(\zeta v))), \nu(|\operatorname{Im}(\zeta v)|^{2m})\}\le 2m. 
\end{equation*}
%If $E$ is a smooth submanifold, the directional derivatives of $\tilde{\Phi}$ along the tangent space of $E$ vanish up to order $2m$, yielding an exact contact order of $\tau(p) = 2m$.???
\end{comment}

\section{The Regularized Maximum Function}\label{rm}
To construct desired smoothly bounded pseudoconvex domains, we shall   truncate the unbounded    domains or extend the local domains obtained in the previous sections  using ellipsoidal  barriers. This requires  patching local defining functions with global bounding structures and  smoothing  the resulting corners. To carry out this procedure  
 without introducing negative eigenvalues that could destroy pseudoconvexity, we shall use the following smooth convex regularized maximum function.

\begin{lemma}\label{mm} 
For any $\epsilon > 0$, there exists a smooth function $M_\epsilon: \mathbb{R}^2 \to \mathbb{R}$ such that the following hold.
\begin{enumerate}
    \item   {Exactness:} $$M_\epsilon(s, t) = \max\{s, t\} \qquad \text{if}\qquad |s - t| \ge \epsilon.$$
    \item  {Lower and upper bounds:} $$  \max\{s, t\} \le M_\epsilon(s, t)\le  \frac{\epsilon}{2} +  \max\{s, t\}\qquad \text{for}\qquad (s, t)\in\mathbb R^2.$$  

\item {Zero-set: } If $M_\epsilon(s,t)=0 $ and $ |s-t|\le \epsilon$, then
$$  -\frac{3\epsilon}{2} \le s \le 0\qquad \text{and}\qquad  -\frac{3\epsilon}{2} \le t \le 0.$$
    
    \item  {Monotonicity:} $$\frac{\partial M_\epsilon}{\partial s} (s, t)\ge 0, \  \  \frac{\partial M_\epsilon}{\partial t}(s, t) \ge 0,\ \   \frac{\partial M_\epsilon}{\partial s}(s, t) +\frac{\partial M_\epsilon}{\partial t}(s, t)  =1 \qquad \text{for}\qquad (s, t)\in \mathbb R^2,$$ and  $$ \frac{\partial M_\epsilon}{\partial s}(s, t) > 0, \ \ \frac{\partial M_\epsilon}{\partial t}(s, t) >0  \qquad \text{if}\qquad |s - t| < \epsilon.$$ 
    
    \item   {Convexity:} The real Hessian matrix $$D^2 M_\epsilon(s, t)  = \phi_\epsilon(s - t) \cdot \begin{pmatrix} 1 & -1 \\ -1 & 1 \end{pmatrix}\qquad \text{for}\qquad (s, t)\in\mathbb R^2$$ for some nonnegative smooth function $\phi_\epsilon: \mathbb R\rightarrow \mathbb R$ with support on $[-\epsilon, \epsilon]$. In particular, the matrix $D^2 M_\epsilon$ is positive semi-definite on $\mathbb R^2$.
\end{enumerate}
\end{lemma}

\begin{proof}
We begin with the classical algebraic identity for the  maximum function:
\begin{equation*}
\max\{s, t\} = \frac{s + t}{2} + \frac{|s - t|}{2}.
\end{equation*}
Let $\phi_\epsilon \in C_c^\infty(\mathbb{R})$ be a standard symmetric mollifier such that $\phi_\epsilon(y) \ge 0$, $\text{supp}(\phi_\epsilon) \subset [-\epsilon, \epsilon]$, $\phi_\epsilon(-y) = \phi_\epsilon(y)$, and \begin{equation}\label{mo}
    \int_{-\epsilon}^\epsilon \phi_\epsilon(y) \, dy = 1.
\end{equation} We define the regularized absolute value function $\chi_\epsilon: \mathbb{R} \to \mathbb{R}$ via convolution  
\begin{equation*} 
\chi_\epsilon(x) = (|x| * \phi_\epsilon)(x) = \int_{-\epsilon}^\epsilon |x - y| \phi_\epsilon(y) \, dy.
\end{equation*}
The regularized maximum function is then defined as 
\begin{equation}\label{md}
M_\epsilon(s, t) : = \frac{s + t}{2} + \frac{1}{2} \chi_\epsilon(s - t).
\end{equation}
Clearly, the function \(M_\epsilon\in C^\infty(\mathbb R^2)\).
\medskip

\noindent\emph{Exactness.} By the definition of $M_\epsilon$, it suffices to show that $\chi_\epsilon (x) = |x|$ outside $(-\epsilon, \epsilon)$. For any $x   \ge \epsilon$ and $y \in [-\epsilon, \epsilon]$ within the support of $\phi_\epsilon$, we have $x - y \ge 0$. Thus 
\begin{equation}\label{123}
\chi_\epsilon(x) = \int_{-\epsilon}^\epsilon (x - y) \phi_\epsilon(y) \, dy = x \int_{-\epsilon}^\epsilon \phi_\epsilon(y) \, dy - \int_{-\epsilon}^\epsilon y \phi_\epsilon(y) \, dy = x=|x|.
\end{equation}
Here we have used the evenness of $\phi_\epsilon$ and the identity \eqref{mo}. 
%Since $\phi_\epsilon$ is an even function, the product $y \phi_\epsilon(y)$ is odd, making the second integral vanish. On the other hand, making use of \eqref{mo} we obtain 
%$$\chi_\epsilon(x) = x = |x|.$$ 
The case when  $x   \le -\epsilon$ is proved similarly by symmetry.  
\medskip

\noindent\emph{Lower and upper  bounds.} 
Since \(\phi_\epsilon(y)dy\) is a probability measure, one has 
\begin{equation*}
\chi_\epsilon(x)
=
\int_{-\epsilon}^{\epsilon}|x-y|\phi_\epsilon(y)\,dy
\ge
\left| \int_{-\epsilon}^{\epsilon}(x-y) \phi_\epsilon(y)\,dy \right|
=
|x|,
\end{equation*}
where the last equality  follows from the same computation as in \eqref{123}. 
 Consequently, by the definition \eqref{md} of $M_\epsilon$,
\begin{equation*}
M_\epsilon(s,t)\ge \frac{s + t}{2} + \frac{1}{2}|s-t| =  \max\{s,t\}.
\end{equation*}
The upper bound follows from the fact that   
$$  \chi_\epsilon(x) \le \int_{-\epsilon}^\epsilon |x|  \phi_\epsilon(y) \, dy + \int_{-\epsilon}^\epsilon |y|  \phi_\epsilon(y) \, dy \le |x| + \epsilon,  $$ and the definition \eqref{md}.

\medskip
\noindent\emph{Zero-set.} This is a direct consequence of the lower and upper bounds above. Indeed, if $M_\epsilon(s,t)=0 $,   then  the   bounds of $M_\epsilon $  imply that 
$$ 0\ge \max\{s,  t \} \ge -\frac{\epsilon}{2}. $$
Suppose first that $\max\{s, t \}  = s $. Then 
$$ 0\ge s \ge -\frac{\epsilon}{2}.$$
On the other hand, by assumption $ |s-t|\le \epsilon$, we have $$s-\epsilon \le t\le s.$$  It then follows that 
$$ 0\ge s\ge t \ge s-\epsilon \ge -\frac{3\epsilon}{2}.$$
The case when   $\max\{s, t \}  =t $ is proved in the same way.

\medskip
\noindent\emph{Monotonicity.}  Differentiating $M_\epsilon(s, t)$ with respect to $s$ and $t$ leads to
\begin{equation} \label{dm}
\frac{\partial M_\epsilon}{\partial s} = \frac{1}{2} + \frac{1}{2}\chi_\epsilon'(s - t), \quad \frac{\partial M_\epsilon}{\partial t} = \frac{1}{2} - \frac{1}{2}\chi_\epsilon'(s - t).
\end{equation}
Thus, we only need to verify that 
\begin{equation}\label{wb}
    -1 \le \chi_\epsilon'(x) \le 1,\ \ \ x\in \mathbb R.
\end{equation}
and 
\begin{equation}\label{lb}
    -1 < \chi_\epsilon'(x) < 1, \ \ \ |x|<\epsilon.
\end{equation}
In fact, differentiating the convolution expression for $\chi_\epsilon(x)$ yields 
\begin{equation*}
\chi_\epsilon'(x) = \frac{d}{dx} \int_{-\epsilon}^\epsilon |x - y| \phi_\epsilon(y) \, dy = \int_{-\epsilon}^\epsilon \text{sgn}(x - y) \phi_\epsilon(y) \, dy.
\end{equation*}
Since $-1 \le \text{sgn}(x - y) \le 1$ and $\phi_\epsilon(y) \ge 0$, \eqref{wb} then follows from  \eqref{mo}. Moreover, \eqref{lb} holds for  $|x|<\epsilon$ because   $\text{sgn}(x - y)$ is neither identically $1$ nor identically $-1$ on the support of $\phi_\epsilon$.  
\medskip

\noindent\emph{Convexity.} 
Recalling  that the distributional derivative of $\text{sgn}(x)$ is $2\delta_0(x)$, where $\delta_0$ denotes the Dirac delta function at the origin, we obtain 
\begin{equation*}
\chi_\epsilon''(x) = \frac{d}{dx}(\text{sgn} * \phi_\epsilon)(x) = (2\delta_0 * \phi_\epsilon)(x) = 2\phi_\epsilon(x),\ \ \ x\in \mathbb R.
\end{equation*}
Since $\phi_\epsilon(x) \ge 0$, we have $\chi_\epsilon''(x) \ge 0$. Applying the chain rule to $M_\epsilon(s, t)$, the real Hessian matrix $  D^2 M_\epsilon $ takes the explicit form 
\begin{equation*}
 D^2 M_\epsilon(s,t) = \begin{pmatrix} \frac{\partial^2 M_\epsilon}{\partial s^2} & \frac{\partial^2 M_\epsilon}{\partial s \partial t} \\ \frac{\partial^2 M_\epsilon}{\partial t \partial s} & \frac{\partial^2 M_\epsilon}{\partial t^2} \end{pmatrix} = \frac{1}{2}\chi_\epsilon''(s - t) \begin{pmatrix} 1 & -1 \\ -1 & 1 \end{pmatrix} = \phi_\epsilon(s - t) \begin{pmatrix} 1 & -1 \\ -1 & 1 \end{pmatrix}.
\end{equation*}
%Testing the quadratic form for an arbitrary real vector $v = (v_1, v_2) \in \mathbb{R}^2$:
%\begin{equation*}
%v^T H v = \phi_\epsilon(s - t) \begin{pmatrix} v_1 & v_2 \end{pmatrix} \begin{pmatrix} v_1 - v_2 \\ -v_1 + v_2 \end{pmatrix} = \phi_\epsilon(s - t)(v_1 - v_2)^2 \ge 0
%\end{equation*}
Since $\phi_\epsilon(s - t) \ge 0$ and the matrix $\begin{pmatrix} 1 & -1 \\ -1 & 1 \end{pmatrix}$ is positive semi-definite, the Hessian matrix   is everywhere positive semi-definite, proving that $M_\epsilon$ is  convex.
\end{proof}

\section{Truncating unbounded models in dimension two}\label{b2}
In this section, we construct smoothly bounded   convex domains of finite D'Angelo type in $\mathbb C^2$ with prescribed weakly pseudoconvex loci.  The construction is obtained by truncating the rigid models from Proposition \ref{r2} with sphere barrier functions of sufficiently large radius and then smoothing the resulting corners using the regularized maximum function. This will prove   Theorem \ref{main1}.
\medskip
 
     Let $ \rho_{\text{loc}}$ denote the defining function    constructed in Proposition \ref{r2}. Namely, 
$$\rho_{\text{loc}}(z, w) = \operatorname{Re} w + y^{2m} + \Phi(x),$$
where $\Phi$ is given by \eqref{di}. We denote  the corresponding rigid cylinder domain     by $\Omega_{\text{loc}}$. 
To construct a   bounded pseudoconvex domain with the prescribed weak locus, we truncate  $\Omega_{\text{loc}} $    with a strictly convex global ball defined by $$\rho_{\text{out}}(z,w)= |z|^2 + |w|^2 - K$$ using the regularized maximum function $M_\epsilon(s, t) = \frac{s+t}{2} + \frac{1}{2}\chi_\epsilon(s-t)$ in Section \ref{rm}. More precisely, we set 
 \begin{equation}\label{pr}\rho(z,w) = M_{\epsilon}\left(\rho_{\text{loc}}(z,w), \rho_{\text{out}}(z,w)\right).\end{equation}
Here,    the positive constants $\epsilon$ and  $K$ are chosen such that 
\begin{equation}
\label{lbK}
0<\epsilon<1\qquad \text{and}\qquad K >  a^2 + \sup_{x \in E} \Phi(x)^2   +4,
\end{equation}
where  the constant $a>0$ is such that the compact set $E \subset [-a, a]$. 
%Here  the positive constant $R$ is chosen large enough such that the target set $E \subset [-a, a]$ sits deep within the interior, satisfying $R^2 \ge a^2 + \epsilon$. 

%The full complex Levi form of the glued domain satisfies the following relation on a complex tangent vector $Z$:
%\begin{equation*}
%L_{\rho}(Z) = \alpha L_{\text{loc}}(Z) + \beta L_{\text{out}}(Z) + H_{\text{max}}\left( \operatorname{Re}\langle \partial\rho_{\text{loc}}, Z \rangle, \operatorname{Re}\langle \partial\rho_{\text{out}}, Z \rangle \right),
%\end{equation*}
%where $\alpha = \frac{\partial M_\epsilon}{\partial s} \ge 0$ and $\beta = \frac{\partial M_\epsilon}{\partial t} \ge 0$. The cross-term contribution simplifies via the algebraic reduction of the smooth maximum's real Hessian:
%\begin{equation*}
%W^T H_{\text{max}} W = \kappa \left( \operatorname{Re}\langle \partial\rho_{\text{loc}} - \partial\rho_{\text{out}}, Z \rangle \right)^2 \ge 0,
%\end{equation*}
%where $\kappa = \frac{1}{2}\chi_\epsilon''(s-t) = \phi_\epsilon(s-t) \ge 0$. This guarantees that the patching operator adds no negative eigenvalues.
\medskip

\noindent\emph{Boundedness.} Making use of  the lower bound of the regularized maximum function in Lemma \ref{mm}, we have 
$$ \rho(z, w) = M_\epsilon(\rho_{\text{loc}}(z,w), \rho_{\text{out}}(z,w))\ge \rho_{\text{out}}(z,w). $$
Therefore  $\Omega =\{(z, w)\in \mathbb C^2:\rho(z, w)<0\}$ is a subset of $ \{(z, w)\in \mathbb C^2: \rho_{\text{out}}(z,w)=|z|^2+|w|^2 -K<0 \}$, which is bounded. 

\medskip

\noindent\emph{Smoothness.}  The smoothness of \(\rho\) follows from the smoothness of \(\rho_{\mathrm{loc}}\), \(\rho_{\mathrm{out}}\), and the regularized maximum function \(M_\epsilon\). To see that $b\Omega$ is smooth, we further show that
\(\nabla\rho\ne 0\) on \(b\Omega\). By the chain rule,
\begin{equation*}
\nabla\rho
=
\alpha\nabla\rho_{\text{loc}}+\beta\nabla\rho_{\text{out}},
\end{equation*}
with  $\alpha: = \frac{\partial M_\epsilon}{\partial s} \ge 0$,  $\beta:= \frac{\partial M_\epsilon}{\partial t} \ge 0$ and $\alpha+\beta=1$ by monotonicity in Lemma \ref{mm}. Suppose for contradiction that \(\nabla\rho=0\) at some point  \(p\in b\Omega\).

If \(\alpha=0\), then \(\beta=1\), and hence \(\nabla\rho_{\text{out}}=0\).  This   forces $ p =(0,0)$. However, by the choice of $\epsilon$ and $K$ in \eqref{lbK},
$$ \rho_{\text{out}}(0) =-K< -\epsilon -\Phi(0)^2-1\le -\epsilon-2|\Phi(0)| \le  -\epsilon +\Phi(0) = -\epsilon +  \rho_{\text{loc}}(0), $$
Thus, by the exactness property in Lemma \ref{mm}, we have
$\rho = \rho_{\text{loc}}$ near $(0, 0)$ where the gradient  is nonzero. This contradicts $\nabla \rho(p)=0$.      On the other hand, if \(\beta=0\), then \(\alpha=1\), and
\(\nabla\rho_{\text{loc}}=0\), which is impossible because $\frac{\partial\rho_{\text{loc}}}{\partial w}=\frac{1}{2}$. Therefore \(\alpha>0\) and \(\beta>0\), so the point lies in the transition
region \(|\rho_{\text{loc}}-\rho_{\text{out}}|<\epsilon\).

The equation \(\nabla \rho =0\) gives, componentwise,
 \begin{equation*}
\begin{split}
    &\alpha + 2\beta u=0;\\
    &\beta v=0;\\
    &    \alpha\Phi'(x) +2\beta x=0;\\
    &\alpha m y^{2m-1} +\beta y=0.
\end{split}
\end{equation*}
Since \(\alpha,\beta, m>0\), the second and the fourth equations give $$v=0,\qquad y=0.$$ The first equation   gives
\begin{equation*}
u=-\frac{\alpha}{2\beta}<0.
\end{equation*}
Multiplying the third equation by \(x\), we
obtain
$\alpha x\Phi'(x)+2\beta x^2=0.$
Since $x\Phi'(x)\ge 0$ by \eqref{pp}, we get \[x=0.\] Thus the hypothetical singular boundary point must be of the form
\begin{equation*}
(0,0,u,0),
\qquad u<0.
\end{equation*}
At such a point,
\begin{equation*}
\rho_{\text{loc}}=u,
\qquad
\rho_{\text{out}}=u^2-K.
\end{equation*}
Making use of the zero-set property in Lemma \ref{mm}, we further have
$$ -\frac{3\epsilon}{2}\le  u\le 0 \qquad\text{and}\qquad  -\frac{3\epsilon}{2}\le  u^2-K\le 0.$$
In particular, since $0<\epsilon<1$, 
$$ K\le \frac{3\epsilon}{2} +u^2 \le \frac{3\epsilon}{2} + \frac{9\epsilon^2}{4}<4,$$
 which contradicts the choice of \(K\).  Hence no such point exists,
and \(b\Omega\) is smooth.

\medskip

\noindent\emph{Convexity.} Since $\rho_{\text{loc}}$ and $\rho_{\text{out}} $ are both convex, the convexity of $\rho$  is a direct consequence of Lemma \ref{cp}. Hence the sub-level set $\Omega$ of $\rho$ is convex. In particular, $\Omega $ is pseudoconvex. 
\medskip

\begin{comment}

To prove the resulting domain $\Omega = \{\rho < 0\}$ is geometrically convex, we examine the real Hessian $D^2 \rho$ of the glued function applied to any vector $v \in \mathbb{R}^4$:
  \begin{align*}v^T (D^2\rho)\  v = &\alpha \left(v^T (D^2\rho_{\text{loc}})\  v\right) +\beta \left(v^T (D^2\rho_{\text{out}}) \ v\right)  +\\
&+\frac{\partial^2 M_{\epsilon}}{\partial s^2}(\nabla \rho_{\text{loc}} \cdot v)^2 + 2\frac{\partial^2 M_{\epsilon}}{\partial s \partial t}(\nabla \rho_{\text{loc}} \cdot v)(\nabla\rho_{\text{out}} \cdot v) + \frac{\partial^2 M_{\epsilon}}{\partial t^2}(\nabla\rho_{\text{out}}\cdot v)^2 \\
 =& \alpha \left(v^T (D^2\rho_{\text{loc}})\  v\right) + \beta \left(v^T (D^2\rho_{\text{out}}) \ v\right)  + W^T(D^2 M_\epsilon) W,
  % D^2 M_\epsilon\left(\nabla\rho_{\text{loc}} \cdot v, \nabla\rho_{\text{out}} \cdot v\right).
  \end{align*}
 with  $W: = \begin{pmatrix} \nabla\rho_{\text{loc}} \cdot v \\ \nabla\rho_{\text{out}}\cdot v\end{pmatrix} $. 
  Since both $\rho_{\text{loc}} $ and $\rho_{\text{out}} $ are convex, and $D^2M_\epsilon$ is  positive semi-definite by Lemma \ref{mm}, one has
  $$ v^T (D^2\rho)\  v\ge 0.$$
 Consequently,   $\Omega$ is convex, and  is automatically pseudoconvex. 
 \end{comment}
  \medskip

  \noindent\emph{Weak locus.} 
To investigate the weak locus of the bounded domain $\Omega$, we partition its total physical boundary $b\Omega $ into three mutually exclusive geometric zones determined strictly by the   sign and magnitude of   $ \rho_{\text{loc}}(z, w) - \rho_{\text{out}}(z, w)$.

%By analyzing the location of the scalar value $\rho_{\text{loc}}(z, w) - \rho_{\text{out}}(z, w)$ on the real line $\mathbb{R} = (-\infty, -\epsilon] \cup (-\epsilon, \epsilon) \cup [\epsilon, \infty)$, the boundary splits into three disjoint components with no missing pieces:

 {Zone 1: The far-field exterior spherical shell.}
%The exterior canopy is defined by the region where the  outer barrier function $\rho_{\text{out}}$ dominates the local cylinder model by a structural threshold margin:
\begin{equation*}
\mathcal{Z}_1 = \{ (z, w) \in b\Omega : \rho_{\text{loc}}(z, w) - \rho_{\text{out}}(z, w) \le -\epsilon\}.
\end{equation*}
By the exactness property of the regularized maximum operator $M_{\epsilon}$ established in Lemma \ref{mm},  the  outer barrier function $\rho_{\text{out}}$ dominates the local cylinder model by the $\epsilon$ margin and hence $$\rho  = \rho_{\text{out}} \qquad \text{on}\ \ \mathcal Z_1.$$ 
The Levi form matches the sphere exactly: $L_{\rho}(Z) = L_{\rho_\text{out}}(Z) > 0$.  This region is strictly pseudoconvex and contains zero weakly pseudoconvex points.

%Consequently, the physical boundary manifold in this zone is given precisely by the level set $\rho_{\text{out}} =  \|x\|^2 + \|y\|^2 + |w|^2 - K = 0$. Substituting this boundary values back into the governing inequality of $\mathcal{Z}_1$ yields:
%\begin{equation*}
%0 - \rho_{\text{loc}}(z, w) \ge \epsilon \implies \rho_{\text{loc}}(z, w) \le -\epsilon
%\end{equation*}
%This condition reveals that the unperturbed local function drops strictly below the smoothing band threshold ($-\epsilon$), locking the partition weights at $\alpha = 0$ and $\beta = 1$. The complex tangent space at these far-field boundary points is completely decoupled from the local potential and inherits the strict, non-degenerate pseudoconvexity of the ellipsoid, preventing any uncontrolled boundary leakage near the cylinder edge. 

 {Zone 2: The non-uniform transition band.} 

\begin{equation*}
\mathcal{Z}_2 = \{ (z, w) \in b\Omega : |\rho_{\text{loc}}(z, w) - \rho_{\text{out}}(z, w)| < \epsilon \}.
\end{equation*}
Inside $\mathcal{Z}_2$, the smoothing operator $M_\epsilon$ blends the local cylinder function and the outer barrier function  within an $\epsilon$-distance of each other,  and $\alpha  = \frac{\partial M_\epsilon}{\partial s} >0$, $\beta = \frac{\partial M_\epsilon}{\partial t} > 0$. 
 For ${\xi}=({\xi}_1, {\xi}_2) \in T_p^{1,0}(b\Omega)$, the Levi form of the domain is
\begin{equation}\label{lm}
\begin{split}
    L_\rho({\xi}, \bar{{\xi}}) &=
\alpha   L_{\rho_{\text{loc}} }({\xi}, \bar{{\xi}}) +  
\beta
 L_{\rho_{\text{out}} }({\xi}, \bar{{\xi}})\\
&\quad+
(M_\epsilon)_{ss}(s,t)\,|\partial s({{\xi}})|^2
+
2\operatorname{Re}
\left(
(M_\epsilon)_{st}(s,t)\,\partial s({{\xi}})\overline{\partial t({{\xi}})}
\right)
+
(M_\epsilon)_{tt}(s,t)\,|\partial t({{\xi}})|^2\\
& = \alpha   L_{\rho_{\text{loc}} }({\xi}, \bar{{\xi}}) +  
\beta
 L_{\rho_{\text{out}} }({\xi}, \bar{{\xi}})+ \left(\overline{\partial s({\xi})}\ \  \overline{ 
\partial t({\xi})}\right)
 D^2M_\epsilon(s,t) 
\begin{pmatrix}
\partial s({\xi})\\[2mm]
\partial t({\xi})
\end{pmatrix}.
\end{split}
\end{equation}
where 
$\partial s({\xi})=\sum_{j=1}^2 s_j {\xi}_j$ and $
\partial t({\xi})=\sum_{j=1}^2 t_j {\xi}_j
$ when viewing ${\xi}$ as a vector field. Since $D^2M_\epsilon$ is positive semi-definite,  the last term is nonnegative. Thus  the total complex Levi form of the domain satisfies $$L_{\rho }  \ge \beta L_{\rho_{\text{out}}} >0.$$
%Because the outer ellipsoidal barrier is strictly convex, its Levi form provides a uniform, strictly positive lower bound, guaranteeing that this entire intermediate blending band remains strictly pseudoconvex and contains zero weak points. Because the outer sphere is strictly convex, its Levi form contribution provides a uniform positive lower bound: $\beta L_{\rho_\text{out}}(X) > 0$. Since the local form and the Hessian cross-terms are strictly non-negative, the total sum is strictly positive: $L_{\rho}(X) \ge \beta L_{\rho_\text{out}}(X) > 0$. 
The transition layer is strictly pseudoconvex and creates zero weakly pseudoconvex points.

 {Zone 3: The deep interior local floor.} %The   region consists of the lower boundary canopy sitting directly over the targeted compact footprint $E \times \{0\}$:
 \begin{equation*}\mathcal{Z}_3 = \{ (z, w) \in b\Omega : \rho_{\text{loc}}(z, w) - \rho_{\text{out}}(z, w) \ge \epsilon \}.
\end{equation*}
In this region, the local cylinder model dominates the outer spherical barrier function by an explicit thickness margin of $\epsilon$. By Lemma \ref{mm}, the regularized maximum reduces to  $$\rho = \rho_{\text{loc}} \qquad \text{on}\ \ \mathcal Z_3.$$ 
Therefore the weak locus $W$ of $\Omega$ satisfies
$$ W = W_{\text{loc}}\cap \mathcal Z_3,$$
where $W_{\text{loc}} $ denotes the weak locus of $\rho_{\text{loc}}$. 

To show that $\Omega$ has the desired weak locus, it suffices to verify that   every point of the form $p = (x_p + i0, -\Phi(x_p) + iv_p)\in W_{\text{loc}}$, where $x_p\in E$ and $|v_p|<1$,  satisfies $\rho_{\text{loc}}  - \rho_{\text{out}}  \ge  \epsilon   $. Indeed, for each such point,  we have 
\begin{equation*}
  \rho_{\text{loc}}(p) - \rho_{\text{out}}(p) = 0 - \left( |x_p|^2 + |\Phi(x_p)|^2 +|v_p|^2 - K \right) = K - |x_p|^2 - \Phi(x_p)^2-|v_p|^2.
\end{equation*}
By the choice of  $K$ in \eqref{lbK},  
\begin{equation*}
K - |x_p|^2 - \Phi(x_p)^2 -|v_p|^2>  K -  a^2 -\sup_{x  \in E} \Phi(x)^2 -1 > \epsilon.
\end{equation*}
  Namely, $p\in \mathcal Z_3$. %  The interior weak locus remains completely unperturbed. 
Consequently, the weakly pseudoconvex locus of $\Omega_{\mathrm{loc}}$ in the region $|\operatorname{Im}w|<1$ remains unchanged in $\Omega$. Moreover, the corresponding D'Angelo type at these points   is carried over to $\Omega$ unchanged. 
\bigskip

The following figure schematically illustrates  the weakly pseudoconvex locus in Theorem~\ref{main1} when $E$ is the  classical middle-thirds Cantor set on $[0,1]$. The weak locus consists of a family of one-dimensional fibers emanating from $E$ in the Reeb direction, forming a ``Cantor forest'' of vertical segments rooted at $E$. The finitely many segments displayed schematically represent the uncountable family of fibers lying over $E$. 

\begin{figure}[!t]
\centering
\begin{tikzpicture}[scale=1.0]

% Level-5 middle-thirds Cantor set on [0.4, 4.6] (width 4.2).
% Each surviving piece has width 4.2/243 = 0.017284.
% 32 left endpoints:
\def\cantorleft{%
  0.000000, 0.041152, 0.123457, 0.164609,
  0.370370, 0.411523, 0.493827, 0.534979,
  1.111111, 1.152263, 1.234568, 1.275720,
  1.481481, 1.522634, 1.604938, 1.646091,
  3.333333, 3.374486, 3.456790, 3.497942,
  3.703704, 3.744856, 3.827160, 3.868313,
  4.444444, 4.485597, 4.567901, 4.609053,
  4.814815, 4.855967, 4.938272, 4.979424}

%\def\cantorleft{%
%  0.40000, 0.45679, 0.56296, 0.61975,
%  0.79210, 0.84889, 0.95506, 1.01185,
%  1.36296, 1.41975, 1.52593, 1.58272,
%  1.75506, 1.81185, 1.91802, 1.97481,
%  2.62519, 2.68198, 2.78815, 2.84494,
%  3.01728, 3.07407, 3.18025, 3.23704,
%  3.58815, 3.64494, 3.75111, 3.80790,
%  3.98025, 4.03704, 4.14321, 4.20000}
% piece width w = 0.017284; center = left + w/2 = left + 0.008642

%============================================================
% LEFT PANEL: Theorem 1.1 -- the "Cantor forest"
%============================================================
\begin{scope}[shift={(0,0)}]

  % axes
  \draw[->] (-0.1,0) -- (6,0) node[right] {\footnotesize $\operatorname{Re} z$};
  \draw[->] (0,-2.3) -- (0,2.3) node[above] {\footnotesize $\operatorname{Im} w$};

  \def\ymin{-2.1}
  \def\ymax{2.1}

  \pgfmathsetseed{2026}

\foreach \a in \cantorleft{
  \pgfmathsetmacro{\treeheight}{0.8 + 1.35*rnd}
  \draw[gray!80,line width=0.4pt]
    (\a+0.010642,-\treeheight)
    --
    (\a+0.010642,\treeheight);
}

% Base set E
\foreach \a in \cantorleft{
  \draw[blue,line width=1pt]
    (\a,-0.03) -- (\a+0.020284,-0.03);
}

% Brace
%\draw[
%  decorate,
%  decoration={brace,amplitude=4pt}
%]
%(-0.12,\ymin) -- (-0.12,\ymax)
%node[midway,left=4pt] {\footnotesize $V$};

  % fibers as thin vertical segments ("trees") over each Cantor piece
 
  \node[blue] at (0.4,0) {\footnotesize $E$};
  \node[below] at (-0.1,0) {\footnotesize $0$};
  \node[below] at (5/3,0) {\footnotesize $\frac{1}{3}$};
  \node[below] at (10/3,0) {\footnotesize $ \frac{2}{3}$};
  \node[below] at (5,0) {\footnotesize $1$};
  \node[align=center] at (2.1,-3)
       {\footnotesize   Weak locus in Theorem~\ref{main1}\\[-2pt]
        \footnotesize $W=\{(z, w)\in b\Omega: \operatorname{Re}z\in E,\ \operatorname{Im}z=0\}\cap \ \overline{\mathcal U}$};

\end{scope}
\end{tikzpicture}

\caption{Schematic representation of the weakly pseudoconvex locus in the $(\operatorname{Re}z,\operatorname{Im}w)$-plane, corresponding to the complex-tangential and Reeb directions, respectively. The condition $\operatorname{Im}z=0$ is suppressed, while $\operatorname{Re}w$ is determined by the defining equation in Proposition~\ref{r2}, and the extent in the $\operatorname{Im}w$-direction is restricted by the neighborhood $\mathcal U$.  The base set $E\subset\mathbb R$ in blue is represented by the middle-thirds Cantor set on $[0,1]$, with five iterations shown.}
\label{fig:cantor-forest}
\end{figure}
\FloatBarrier

 \section{Constructing convex models in dimension two, continued}\label{b2'}
 
Given a compact set $E\subset \mathbb R$, we construct in this section a smoothly bounded convex domain in $\mathbb C^2$ whose weak locus is   diffeomorphic to $E$.  This will  prove Theorem \ref{main1'}. Write  $ (z,w)=(x+iy,u+iv)\in \mathbb C^2.$ After a translation in the \(x\)-variable, we may assume without loss of generality that \(0\notin E\). 
\medskip

To eliminate the tree-like components over the base points in $E$ in Theorem \ref{main1}, we modify the local rigid model by adding a controlling term  in the $v$ direction. More precisely, for  fixed integer \(N\ge m\ge 2\), we define the local model by
\begin{equation*}
\rho_{\text{loc}}(z,w)
=
u+v^{2N}+y^{2m}+\Phi(x),
\end{equation*}
where $\Phi$ is defined as in  \eqref{di}.
The strictly convex global barrier and the patched defining function stays the same:  
\begin{equation*}
\rho_{\text{out}}(z,w)
=
x^2+y^2+u^2+v^2-K,
\end{equation*}
where \(K>0\) is chosen sufficiently large as in \eqref{lbK}, and 
\begin{equation*}
\rho(z,w)
=
M_\epsilon\bigl(\rho_{\text{loc}}(z,w),\rho_{\text{out}}(z,w)\bigr).
\end{equation*}
Set
\begin{equation*}
\Omega=\{(z,w)\in\mathbb C^2:\rho(z,w)<0\}.
\end{equation*}
The boundedness of    \(\Omega\) follows from  $
\rho_{\text{out}}\le \rho< 0. $  Since both $\rho_{\text{loc}}$ and $\rho_{\text{out}}$ are convex, the convexity of $\Omega$ then follows from Lemma \ref{cp}.  In particular, $\Omega$ is  pseudoconvex.
 \medskip

\noindent\emph{Smoothness.}   Suppose that 
\begin{equation*}
\nabla\rho
=
\alpha\nabla\rho_{\text{loc}}+\beta\nabla\rho_{\text{out}} 
\end{equation*}
is zero at some point of \(b\Omega\). This can only happen when both \(\alpha>0\) and \(\beta>0\) as argued in   Section \ref{b2}, so the point lies in the transition
region \(|\rho_{\text{loc}}-\rho_{\text{out}}|<\epsilon\). In this case, by inspecting the components of $\nabla \rho$, one has 
\begin{equation*}
    \begin{split}
      & \alpha\,N v^{2N-1}+\beta v =0;\\
&\alpha\,m y^{2m-1}+\beta y=0;\\
&\alpha\,\Phi'(x)+2\beta x=0;\\
&\alpha+2\beta u =0. 
    \end{split}
\end{equation*}
With a similar argument as in Section \ref{b2}, one can obtain that   any hypothetical singular boundary point must be of the form
\begin{equation*}
(0,0,u,0),
\qquad u<0.
\end{equation*}
At such a point,
\begin{equation*}
\rho_{\text{loc}}=u,
\qquad
\rho_{\text{out}}=u^2-K.
\end{equation*}
   This is
impossible since   \(K\) is chosen  sufficiently large as in \eqref{lbK}.  Hence   \(b\Omega\) is smooth.
 \medskip

\noindent\emph{Weak locus.}  The Levi form takes a similar form as in \eqref{lm}.  Hence  the weakly pseudoconvex points lie in the region where $\beta =0$ and $\alpha =1$, %equivalently, $\chi_\epsilon'( \rho_{\text{loc}} - \rho_{\text{out} }) =1 $ by \eqref{dm}. It follows further by  \eqref{lb} that 
where  $\rho_{\text{loc}} - \rho_{\text{out} }\ge \epsilon $, and
\begin{equation*}
\rho=\rho_{\text{loc}}.
\end{equation*}
%Thus all weakly pseudoconvex points of the patched domain occur in the region where the local model governs the boundary.
On this region, the weakly pseudoconvex locus is determined by the Levi form 
of \(\rho_{\text{loc}}\): 
\begin{equation}\label{lf}
    L(\xi, \bar \xi) =\frac14 |\xi_1|^2\left(\psi(x)+2m(2m-1)y^{2m-2}\right) + \frac12|\xi_2|^2\, N(2N-1)v^{2N-2}   
\end{equation}
for every nonzero vector $\xi= (\xi_1, \xi_2)\in T_p^{1,0}(b\Omega)  $ that satisfies   
\begin{equation}\label{xi}
   \xi_1 ( \Phi'(x) -2imy^{2m-1}  )+\xi_2 (1-2iNv^{2N-1}) =0.  
\end{equation}
  We first note that $$\xi_1\ne 0.$$ Otherwise, since the coefficient $ 1-2iNv^{2N-1} $  is never zero in \eqref{xi}, it follows that  $\xi_2=0$. This contradicts  the nonzero assumption of $\xi$. 
  
  Since  the Levi form  \eqref{lf} is a sum of nonnegative terms, its vanishing forces each term in the sum to vanish. Therefore, using the construction of $\psi$ and the fact that $\xi_1\ne 0$, we obtain   
\begin{equation*}
x\in E,
\qquad
y=0,
\qquad
\xi_2 v=0. 
\end{equation*}
On the other hand, since the nonnegative function  $\psi$ is strictly positive off $E$ and $0\notin E$, $$ \Phi'(x) = \int_0^x \psi(t)dt\ne 0\ \ \  \text{for all}\ \ x\in E. $$ This, together with \eqref{xi}, implies $\xi_2\ne 0$. Thus
$$v=0. $$
Hence, the weakly pseudoconvex locus of the local model is precisely the prescribed set. By the choice of $K$ 
in \eqref{lbK}, together with a similar argument at the end of the previous  section for zone 3,  these weakly pseudoconvex points are carried over from $\Omega_{\text{loc}}$ to $\Omega$. The desired weak locus is thus verified. 
  
\begin{comment}
In the general case, additional weak points may occur over points where
\(\Phi'(x)=0\).  Indeed, because
\begin{equation*}
\Phi'(x)=\int_0^x\psi(t)\,dt,
\end{equation*}
the vanishing of \(\Phi'\) is governed by the topology of the intervals on
which \(\psi\) vanishes.  For example, if \(E\) is totally disconnected and
contains the origin, then
\begin{equation*}
\Phi'(x)=0
\qquad\Longleftrightarrow\qquad
x=0.
\end{equation*}
In that case one obtains an additional weak segment over the origin:
\begin{equation*}
W_2
=
\{(0,-v^{2N}+iv): |v|<\sqrt{R^2-\epsilon}\}.
\end{equation*}
Thus the weakly pseudoconvex locus of the global model consists of the
prescribed closed set together with possible one-dimensional segments in the
\(w\)-fiber over the roots of \(\Phi'\).
\end{comment}

 \medskip

\noindent\emph{D'Angelo type.}     Fix a point
 $ p= (z_p, w_p) = (x_p+i0, -\Phi(x_p) +i0)\in W$  with  \(x_p\in E\).  We use the biholomorphic change of coordinates
\begin{equation*}
\begin{split}
&{\tilde z}=z-x_p;\\
&\widetilde w=w-w_p+\Phi'(x_p)(z-x_p),
\end{split}
\end{equation*}
which sends \(p\) to the origin.  In these coordinates, the defining
function has the local form
\begin{equation*}
\widetilde\rho({\tilde z},\widetilde w)
=
\operatorname{Re}\widetilde w
+
\left(\operatorname{Im}\widetilde w
-\Phi'(x_p)\operatorname{Im}{\tilde z}\right)^{2N}
+
(\operatorname{Im}{\tilde z})^{2m}
+
R(\operatorname{Re}{\tilde z}),
\end{equation*}
where \(R\) is flat at the origin.

Let
 $\gamma(\zeta)=({\tilde z}(\zeta),\widetilde w(\zeta))$ 
be a nonconstant holomorphic curve with \(\gamma(0)=0\). Pulling back $\widetilde\rho$ along $\gamma$, we obtain
$$
\begin{aligned}
(\widetilde\rho\circ\gamma)(\zeta)
&=
\operatorname{Re}\widetilde w(\zeta)   +
\left(
\operatorname{Im}\widetilde w(\zeta)-
\Phi'(x_p)
\operatorname{Im}\tilde z(\zeta)
\right)^{2N} +
(\operatorname{Im}\tilde z(\zeta))^{2m}
+
R(\operatorname{Re}\tilde z(\zeta)).
\end{aligned}
$$
First suppose that $\nu(z(\zeta))\ge \nu(w(\zeta))$. Then
$\operatorname{Re}\widetilde w(\zeta)$ has vanishing order $\nu(w(\zeta))  $, while all other terms have order strictly larger than $\nu(w(\zeta))$. Hence
$ \nu(\widetilde\rho\circ\gamma)= \nu(w(\zeta)),$  and 
$$
\frac{\nu(\widetilde\rho\circ\gamma)}
{\nu(\gamma)}=1.
$$

Now suppose that $\nu(z(\zeta))< \nu(w(\zeta))$. Then $\nu(\gamma)=\nu(z(\zeta))$.  Moreover,
$(\operatorname{Im}\tilde z(\zeta))^{2m}$ has vanishing order $2m\nu(z(\zeta))$ and $ 
R(\operatorname{Re}\tilde z(\zeta))
$ 
vanishes to infinite order. Similarly as in the proof of Proposition \ref{lc}, let $P_d$ be the first nonzero homogeneous term in the Taylor expansion of
$
\left(
\operatorname{Im}\widetilde w(\zeta) -
\Phi'(x_p)\operatorname{Im}\tilde z(\zeta)
\right)^{2N}
+
(\operatorname{Im}\tilde z(\zeta))^{2m}.
$
Then $P_d$ is nonnegative and not constantly zero, with  degree
$$
d\le 2m \nu(z(\zeta)).
$$
Since $P_d(0)=0$, the maximum principle implies that $P_d$ cannot be harmonic. On the other hand, every homogeneous term of
$\operatorname{Re}\widetilde w(\zeta)$ is harmonic. Thus
$\operatorname{Re}\widetilde w(\zeta)$ cannot cancel $P_d$.
Therefore
$$
\nu(\widetilde\rho\circ\gamma)
\le d
\le  2m \nu(z(\zeta)).
$$
and 
$$
\frac{\nu(\widetilde\rho\circ\gamma)}
{\nu(\gamma)}
\le
2m.
$$
Moreover, this bound is sharp by letting the curve
 be $
\gamma(\zeta)=(\zeta,0),
$ which 
gives
\begin{equation*}
\nu(\widetilde\rho\circ\gamma)=2m.
\end{equation*}
Hence the D'Angelo type along the prescribed weak locus is exactly \(2m\).

\section{Extending  local models in higher dimensions}\label{bn}
In this section, we prove Theorems \ref{main2} and \ref{main3} by smoothly extending the local hypersurfaces obtained in Propositions  \ref{lp} and \ref{lc} using an outer ellipsoidal barrier and the regularized maximum function. 
\medskip

Fix $x_0\in E\subset \mathbb R^{n-1}$ and $m\ge 2$. Let $\rho_{\text{loc}}$ be  the local defining function  constructed in Proposition   \ref{lp},  defined on $U\times \mathbb C$ and  denote the corresponding local hypersurface  by $M_{\text{loc}}$. Namely,
$$ \rho_{\text{loc}}(z,w) = \operatorname{Re}w+f(z, \bar z) $$
with $f$ satisfying \eqref{pnf}. 
Without loss of generality, we may assume that  the neighborhood $U: =B_{\delta}(x_0+i0)$,   the ball centered at $x_0+i0$ with radius $\delta>0$ in $\mathbb C^{n-1}$.  For $\epsilon>0$,   we shall patch the local defining function with the strictly convex ellipsoidal barrier:
$$\rho_{\text{out}}(z, w) =  A|z-x_0|^2 + |w|^2 - K, $$
where   $A$ and $K $ are  constants satisfying
\begin{equation}\label{a}
    A> \max\left\{\frac{32( \sup_{z\in U}|f(z, \bar z)|^2 + \epsilon)}{\delta^2}, \frac{4(\sup_{z\in U}|f(z, \bar z)|  +2\epsilon ) \sup_{z\in U}|\nabla f(z, \bar z)|}{\delta},  1 \right\}
\end{equation}
and 
\begin{equation}\label{kn}
   \frac{A\delta^2}{4}  > K> \frac{A\delta^2}{8} +\sup_{z\in U}|f(z, \bar z)|^2 +\epsilon.
\end{equation}
Such a choice of $K$ is possible by \eqref{a}. In particular, one has \begin{equation}\label{ep} 
  \frac{A\delta^2}{4}   >K> \frac{A\delta^2}{8}\qquad\text{and}\qquad 0<\epsilon <\frac{A\delta^2}{32}.
\end{equation}  
As will be seen later, the upper bound of $K$ in \eqref{kn} is to confine the desired global domain within a cylinder where $\rho_{\text{loc}}$ is defined, and the lower bound there is to preserve the weak locus of $M_{\text{loc}}$ near $(x_0+i0, -f(x_0, x_0)+i0)$. 

We define, on \(U\times\mathbb C\), our   defining function again via the regularized maximum function $M_\epsilon$:
\begin{equation*}
\rho(z, w) = M_\epsilon(\rho_{\text{loc}}(z, w), \rho_{\text{out}}(z, w)), 
\end{equation*}
and set
$$
\Omega
=\{(z,w)\in U\times\mathbb C:\rho(z,w)<0\}. 
$$
%and
%$$ M = {(z,w)\in U\times\mathbb C:\rho(z,w)=0}.$$
%Although $\rho_{\mathrm{loc}}$ is only locally defined,
We claim that the set $ {\Omega}$ is compactly contained in the region $ U\times\mathbb C$,  where $\rho_{\mathrm{loc}}$ is well defined. Indeed, by Lemma \ref{mm},  we have    
$$
\rho_{\mathrm{out}}(z,w)\le\rho(z,w) \le  0  \qquad \text{on}\qquad \overline{\Omega}.
$$
Since
$ 
\rho_{\mathrm{out}}(z,w)=
A|z-x_0|^2+|w|^2-K,
$ 
it follows that
\begin{equation*} 
    A|z-x_0|^2+|w|^2\le K.
\end{equation*}
In particular,
$$
|z-x_0|
\le
\sqrt{\frac{K}{A}}<
\frac{\delta}{2} \qquad\text{and}\qquad |w| \le \sqrt K.
$$
by the choice of $K$ in \eqref{kn}. 
Therefore
$ 
z\in B_{\delta/2}(x_0+i0)
\Subset
B_\delta(x_0+i0)=U.
$ 
The claim is thus proved. Consequently,  $\overline{\Omega}$ cannot meet the boundary of the local coordinate cylinder, and   $\Omega$ is a well-defined bounded domain in $\mathbb C^n$.

%Moreover, since $\nabla\rho\ne 0$ on $M$, the implicit function theorem implies that $$ M=b\Omega $$ is a smooth hypersurface. 
%Since $\overline{\Omega}$ is compactly contained in $U\times\mathbb C$, no boundary component is introduced from the boundary of the local chart.

 To see $b\Omega$ is smooth, we need to verify that $\nabla \rho$ is nonvanishing on $\rho=0$. Again, since 
\begin{equation*}
\nabla\rho
=
\alpha\nabla\rho_{\text{loc}}+\beta\nabla\rho_{\text{out}},
\end{equation*}
with $\alpha\ge 0, \beta\ge 0$ and $\alpha+\beta=1$, it suffices to only  consider the region when both \(\alpha>0\) and \(\beta>0\), namely,  the transition region  \(|\rho_{\text{loc}}-\rho_{\text{out}}|\le\epsilon\). On this region, if \(\rho=0\), then it follows from  the zero-set property in Lemma \ref{mm} that 
\begin{equation}\label{bo}
     |\rho_{\text{loc}}|< 2\epsilon, \qquad |\rho_{\text{out}}| <2\epsilon. 
\end{equation}
  Let 
\[ S=\sup_{z\in U}|f(z, \bar z)|, \qquad L=\sup_{z\in U}|\nabla f(z, \bar z)|. \] By the bounds for $ \rho_{\text{loc}} $ and $ \rho_{\text{out}} $ in \eqref{bo}, together with  their explicit  expressions,  we obtain 
\begin{equation}\label{u}
    |u|\le S+2\epsilon, 
\end{equation} 
and  
\[  A|z-x_0|^2+|w|^2-K\ge -2\epsilon.   \]  
Using the  inequalities \eqref{ep} for $K$  and $\epsilon$ , we obtain 
\begin{equation}\label{11}
A|z-x_0|^2+v^2 \ge K -2\epsilon> \frac{A\delta^2}{16}.     
\end{equation} 
Suppose, for contradiction, that \(\nabla\rho=0\) at  a point on $b\Omega$. 
 Since \(\beta>0\) and the \(v\)-component of \(\nabla\rho=0\) gives \( 2\beta v=0, \) we have  \[v=0.\] 
 Plugging the above into \eqref{11} further yields 
   \begin{equation}\label{z}
       |z-x_0| > \frac{\delta}{4}.  
   \end{equation} 
   The \(u\)-component of \(\nabla\rho=0\) gives \[ \alpha+2\beta u=0. \] Hence by \eqref{u}
  \begin{equation}\label{al}
      \alpha\le 2\beta |u| \le 2\beta(S+2\epsilon).  
  \end{equation} 
   The \(z\)-components of \(\nabla\rho=0\) give \[ \alpha \nabla f(z, \bar z)+2\beta A(z-x_0)=0. \] Taking norms and dividing by $2\beta$, we get from \eqref{al} that 
     \[ A|z-x_0| \le \frac{\alpha }{2\beta}L \le (S+2\epsilon)L. \] Combining this further with  \eqref{z}, we obtain 
     \[\frac {A \delta}{4} < (S+2\epsilon)L. \] 
     This contradicts with \eqref{a}. Therefore \(\nabla\rho\neq 0 \) on  \(b\Omega\), and the boundary \(b\Omega\) is smooth.

The verification of   pseudoconvexity  of the domain $\Omega$ is the same  as in Section \ref{b2} in dimension $n=2$, and will be  omitted. It remains only  to determine  the weak locus. As before,  we divide into three regions according to the value of $\rho_{\text{loc}} - \rho_{\text{out}}$.   

Zone 1:  The portion dominated by the  outer barrier function $\rho_{\text{out}}$:
\begin{equation*}
\mathcal{Z}_1 = \{ (z, w) \in b\Omega : \rho_{\text{loc}}(z, w) - \rho_{\text{out}}(z, w) \le -\epsilon\},
\end{equation*}
where   $$\rho  = \rho_{\text{out}} \qquad \text{on}\ \ \mathcal Z_1.$$ 
The Levi form matches the sphere exactly: $L_{\rho}(Z) = L_{\rho_\text{out}}(Z) > 0$.  This region is strictly pseudoconvex and contains zero weakly pseudoconvex points.

%Consequently, the physical boundary manifold in this zone is given precisely by the level set $\rho_{\text{out}} =  \|x\|^2 + \|y\|^2 + |w|^2 - K = 0$. Substituting this boundary values back into the governing inequality of $\mathcal{Z}_1$ yields:
%\begin{equation*}
%0 - \rho_{\text{loc}}(z, w) \ge \epsilon \implies \rho_{\text{loc}}(z, w) \le -\epsilon
%\end{equation*}
%This condition reveals that the unperturbed local function drops strictly below the smoothing band threshold ($-\epsilon$), locking the partition weights at $\alpha = 0$ and $\beta = 1$. The complex tangent space at these far-field boundary points is completely decoupled from the local potential and inherits the strict, non-degenerate pseudoconvexity of the ellipsoid, preventing any uncontrolled boundary leakage near the cylinder edge. 

  Zone 2:  The transition layer between the local hypersurface and the barrier sphere: 
\begin{equation*}
\mathcal{Z}_2 = \{ (z, w) \in b\Omega : |\rho_{\text{loc}}(z, w) - \rho_{\text{out}}(z, w)| < \epsilon \}
\end{equation*}
Investigating  the Levi form with respect to  $\rho$ in exactly the same way on as in $n=2$ case in Section \ref{b2}, we can verify that the  transition layer is strictly pseudoconvex and creates zero weakly pseudoconvex points.
%Because the outer ellipsoidal barrier is strictly convex, its Levi form provides a uniform, strictly positive lower bound, guaranteeing that this entire intermediate blending band remains strictly pseudoconvex and contains zero weak points. Because the outer sphere is strictly convex, its Levi form contribution provides a uniform positive lower bound: $\beta L_{\rho_\text{out}}(X) > 0$. Since the local form and the Hessian cross-terms are strictly non-negative, the total sum is strictly positive: $L_{\rho}(X) \ge \beta L_{\rho_\text{out}}(X) > 0$. 

  Zone 3: The portion dominated by $\rho_{\text{loc}}$: 
  \begin{equation*}\mathcal{Z}_3 = \{ (z, w) \in b\Omega : \rho_{\text{loc}}(z, w) - \rho_{\text{out}}(z, w) \ge \epsilon \},
\end{equation*}
where  $$\rho  = \rho_{\text{loc}} \qquad \text{on}\ \ \mathcal Z_3.$$
Therefore the weak locus $W$ of $\Omega$ satisfies
$$ W = W_{\text{loc} }\cap \mathcal Z_3,$$
where $W_{\text{loc}} $ denotes the weak locus of $M_{\text{loc}}$. 

It then suffices to verify  that   a neighborhood of  $ p_0 = (x_0+i0, -f(x_0, x_0)+i0)$ in $M_{\text{loc}}$ is  contained in $ \mathcal Z_3$.   Let $p=(z_p, w_p)\in M_{\text{loc}}$,  with  $z_p\in B_{\frac{\delta}{3}}(x_0+i0)$, $w_p = -f(z_p, \bar z_p)+iv_p $ and $|v_p|<\frac{\delta}{9} $.  Then    $\rho_{\text{loc}}(p)=0$. On the other hand,  using the lower bound of $K$ in \eqref{kn}, we obtain   
\begin{equation*}
\begin{split}
   \rho_{\text{out}}(p)=  A|z_p-x_0|^2+|w_p|^2-K 
   \le& \  \frac{A\delta^2}{9} +\sup_{z\in U}|f(z, \bar z)|^2 +   \frac{\delta^2}{81} -K\\
   <&\  \frac{A\delta^2}{8} +\sup_{z\in U}|f(z, \bar z)|^2 -K <-\epsilon.
\end{split}
\end{equation*}
Consequently,    $$\rho_{\text{loc}}(p) - \rho_{\text{out}}(p)>  \epsilon.$$ Thus $p\in \mathcal Z_3$. Namely,  a neighborhood of \(p_0\) in \(M_{\mathrm{loc}}\) is contained in \(\mathcal Z_3\), and on this neighborhood the global hypersurface \(b\Omega\) agrees with the local hypersurface \(M_{\mathrm{loc}}\).
Therefore, the corresponding D'Angelo type is carried over to $b\Omega$ unchanged. 

\medskip
   
 For the case  $E\subset \mathbb R^{n-2}$,  Proposition  \ref{lc} gives rise to a local hypersurface  $M_{\text{loc}}$  defined on $(U+i\mathbb R^{n-1})\times \mathbb C\subset \mathbb C^n$, where $U$ is a neighborhood  of $(\hat x_0, 0)$ in $\mathbb R^{n-1}$. A careful inpsection of the proof shows that   the tube neighborhood $U+i\mathbb R^{n-1}$ in Proposition  \ref{lc} may be replaced by a ball $B_\delta((\hat x_0, 0)+i0)\subset \mathbb C^{n-1}$ without affecting the  conclusion of the proposition. Therefore    the  patching argument and verification carried out  above  for Proposition \ref{lp} would apply equally  to the case  $E\subset \mathbb R^{n-2}$ in Proposition  \ref{lc}, with $x_0$ replaced by $(\hat x_0, 0)$. We therefore omit the details.

\appendix

\section{Proof of Lemmas \ref{ld1}, \ref{covering} and \ref{cp}}\label{ap}
We first prove Lemma \ref{ld1}, which relates  the determinant of the bordered Hessian  to the Levi determinant.

\begin{proof}[Proof of Lemma \ref{ld1}. ]
%To evaluate the determinant cleanly, we perform a unitary change of basis in $\mathbb{C}^n$ that aligns the boundary's complex normal direction with the $n$-th coordinate axis. 
Choose an orthonormal basis $\{e_1, \dots, e_{n-1}\}$ for the $(n-1)$-dimensional complex tangent space $T_p^{1,0}(b\Omega)$.
Define the $n$-th basis vector as the normalized complex normal 
\[
e_n = \frac{1}{|\partial\rho|} (\bar{\partial}\rho)^T.
\]
Construct the $n \times n$ unitary matrix $U = [e_1, e_2, \dots, e_n]$, where $e_k$ are column vectors. 
 
Under the holomorphic coordinate change $z = U{\tilde z}$, we define the localized defining function $\tilde{\rho}({\tilde z}) = \rho(U{\tilde z})$. In particular, $$\det(H^{\text{bd}}_\rho) = \det(\tilde H^{\text{bd}}_{\tilde \rho}),$$
where $H^{\text{bd}}_\rho$ and $\tilde H^{\text{bd}}_{\tilde \rho}$ are the corresponding  bordered Hessian matrix with respect to coordinates $  z$ and $\tilde z$, respectively. 
By the  chain rule, 
\[
\frac{\partial\tilde{\rho}}{\partial {\tilde z}_k} = \sum_{j=1}^n \frac{\partial \rho}{\partial z_j} \frac{\partial z_j}{\partial {\tilde z}_k}  = \sum_{j=1}^n \frac{\partial\rho}{\partial z_j} (e_k)_j =0,
\]
where the last equality is due to the fact that the vector $e_k $ lies in the complex tangent space $T_p^{1,0}(b\Omega)$. For the normal direction $k=n$, substituting the components of $e_n$ yields 
\[
\frac{\partial\tilde{\rho}}{\partial {\tilde z}_n} = \sum_{j=1}^n \frac{\partial\rho}{\partial z_j} \cdot  \frac{1}{|\partial\rho|} \frac{\partial \rho}{\partial \bar{z}_j}  = \frac{1}{|\partial\rho|} \sum_{j=1}^n \left| \frac{\partial\rho}{\partial z_j} \right|^2 = |\partial\rho|.
\]
Thus, in the ${\tilde z}$-coordinates, the complex gradient row vector is exactly 
\[
\partial_{\tilde z} \tilde{\rho} = (0, 0, \dots, 0, |\partial\rho|).
\]
%Since $\rho$ is real-valued, conjugating the gradient yields the identical anti-holomorphic row vector:
%\[(\bar{\partial}_{\tilde z} \tilde{\rho}) = (0, 0, \dots, 0, |\partial\rho|).\]
On the other hand, by construction and definition, the upper-left $(n-1) \times (n-1)$ block of the transformed Hessian $H_\rho$ is precisely the restricted Hessian $\tilde{H}_\rho$.  Therefore, $\tilde H^{\text{bd}}_{\tilde \rho}$ takes the following form 
\[
\tilde H^{\text{bd}}_{\tilde \rho} = \begin{pmatrix} 0 & 0 & \dots & 0 & |\partial\rho| \\ 0 & & & & * \\ \vdots & & \tilde{H}_\rho & & \vdots \\ 0 & & & & * \\ |\partial\rho| & * & \dots & * & * \end{pmatrix}.
\]
We then obtain by cofactor expansion that 
\[
\det(\tilde H^{\text{bd}}_{\tilde \rho})   = -|\partial\rho|^2 \det(\tilde{H}_\rho) =  -|\partial\rho|^2 \prod_{j=1}^{n-1} \lambda_j.
\]
The proof is complete. 
\end{proof}

The following lemma provides   the construction of a locally finite covering for open sets of $\mathbb R^n$, as used in Proposition \ref{Wf}. 

\begin{lemma}\label{covering}
    Let $U$ be an open set in $\mathbb R^d$. For any $r>0$, there exists a countable, locally finite collection of open balls with radius less than $r$ that covers $U$.
\end{lemma}
 
\begin{proof}
       For each integer $m \ge 1$, define the compact set
    $$ K_m = \{x \in U : \|x\| \le m \text{ and } d(x, U^c) \ge 1/m \},$$
    where  $d(x, A)$ denotes the distance from   $x$ to the set $A$.  Then   $K_m \subset \text{int}(K_{m+1})$ and $\bigcup_{m=1}^\infty K_m = U$. 
     Let $K_0 = \emptyset$ and set  
    $$ V_m = \text{int}(K_{m+1}) \setminus K_{m-1} \quad \text{for } m \ge 1. $$
    The collection of open sets $\{V_m\}_{m=1}^\infty$ is a locally finite  covering of $U$. 

    Now, since each $\overline{V_m}$ is compact,   $\overline{V_m}$ can be covered by finitely  many open balls of radius less  than $r$, each contained in $K_{m+2}\setminus \overline{K_{m-2}}$ (setting $K_{-1} = \emptyset$). Denote this finite collection     of balls by $\mathcal{B}_m$.
  Set  
    $$ \mathcal{B} = \bigcup_{m=1}^\infty \mathcal{B}_m.$$
    Then $  \mathcal{B} $ is a countable collection of open balls with radius less than $r$, and its  union is $U$. Indeed, every point of $U$ has a neighborhood intersecting only finitely many of the sets $V_m$, and hence only finitely many of the corresponding collections $\mathcal B_m$. Thus $\mathcal B$ is the desired locally finite covering of $U$.
\end{proof}

 The following is a well-known fact concerning  the composition of smooth non-decreasing convex functions with convex or plurisubharmonic functions.

\begin{lemma} \label{cp}
Let \(\Psi:\mathbb R^N\to\mathbb R\) be a smooth convex function that is nondecreasing in each variable. 
\begin{enumerate}
    \item If \(D\subset\mathbb R^d\) is a convex domain, and  \(u=(u_1,\ldots,u_N):D\to\mathbb R^N \) is smooth, with each \(u_j\) convex on $D$, then \(\Psi\circ u\) is convex on \(D\). 

    \item If \(\Omega\subset\mathbb C^d\) is a domain,  and \(u=(u_1,\ldots,u_N):\Omega\to\mathbb R^N \) is smooth, with each \(u_j\) plurisubharmonic on $\Omega$, then \(\Psi\circ u\) is plurisubharmonic on $\Omega$.
\end{enumerate}
  \end{lemma}

\begin{proof} We first prove the convex case. Since \(\Psi\) is smooth and nondecreasing in each variable, we have \[ \Psi_j\ge 0,\qquad j=1,\ldots,N, \] where \(\Psi_j=\partial \Psi/\partial x_j\). The convexity of  \(\Psi\) further gives its real Hessian  \[ D^2\Psi\ge 0. \] Let \(U=\Psi\circ u\). By the chain rule, \[ D^2U = \sum_{j=1}^N \Psi_j(u)\,D^2u_j + \sum_{j,k=1}^N \Psi_{jk}(u)\,du_j\otimes du_k. \]
Thus, for every \({\xi}\in\mathbb R^d\), \[ D^2U({\xi},{\xi}) = \sum_{j=1}^N \Psi_j(u)\,D^2u_j({\xi},{\xi}) + \sum_{j,k=1}^N \Psi_{jk}(u)\,du_j({\xi})\,du_k({\xi}). \] The first sum is nonnegative because \(\Psi_j\ge 0\) and each \(u_j\) is convex. The second sum is nonnegative because \(D^2\Psi\ge 0\). Hence \(D^2U\ge 0\), and therefore \(U=\Psi\circ u\) is convex. 

The plurisubharmonic case is analogous. Let now \(\Omega\subset\mathbb C^d\), and assume that each \(u_j\) is smooth plurisubharmonic. For every \({\xi}\in\mathbb C^d\), the complex Hessian of \(U=\Psi\circ u\) is
\[ H_U({\xi, \overline{\xi}}) = \sum_{j=1}^N \Psi_j(u)\,H_{u_j}({\xi}) + \sum_{j,k=1}^N \Psi_{jk}(u)\,\partial u_j({\xi})\, \overline{\partial u_k({\xi})}. \] 
Since both terms are nonnegative by assumption, 
%each \(u_j\) is plurisubharmonic, its complex Hessian \(H_{u_j}({\xi})\ge 0\). Since \(\Psi\) is nondecreasing in each variable, \(\Psi_j\ge 0\). Hence the first sum is nonnegative. The second sum is also nonnegative because \(D^2\Psi\ge 0\). Therefore 
\[ H_U({\xi, \overline{\xi}})\ge 0 \] for every \({\xi}\in\mathbb C^d\). Thus \(U=\Psi\circ u\) is plurisubharmonic. \end{proof}

\section{Remarks on the conjectures}\label{sac}

We begin by proving a relationship between   the not-flatness of the Levi determinant and the surface measure of  the weakly pseudoconvex locus. In particular, this  shows that Conjecture 2 implies Conjecture 1 in the introduction.

\begin{proposition}\label{2to1}
Let $\Omega\subset\mathbb C^n$ be a smooth pseudoconvex domain. Assume that, at every weakly pseudoconvex boundary point, the Levi determinant does not vanish to infinite order as a smooth function on \(b\Omega\). Then the weakly pseudoconvex locus \(W\) of $\Omega$ has zero surface measure on \(b\Omega\).
\end{proposition}

The proposition follows directly from the following elementary lemma, applied in local boundary coordinates to the Levi determinant $ \Lambda_\rho$.

\begin{lemma}\label{ele}
    Let $h\in C^\infty(V)$, where
$V\subset\mathbb R^d$ is open. If $h$ does not vanish to infinite order at any
point of $h^{-1}(0)$, then $h^{-1}(0)$ has Lebesgue measure zero in
$\mathbb R^d$.

\end{lemma}
\begin{proof}
    Suppose by contradiction that $E: =h^{-1}(0)$ has positive measure.
By the Lebesgue's density theorem, there exists a point \(x_0\in E\) which
is a density point of \(E\). %; that is,
%\begin{equation*}
%\lim_{r\to 0} \frac{|E\cap B_r(x_0)|}{|B_r(x_0)|} =1.
%\end{equation*}
Equivalently, 
\begin{equation*}
\lim_{r\to 0}
\frac{|B_r(x_0)\setminus E|}{|B_r(x_0)|}
=0.
\end{equation*}

Let $k$ be the  least integer  such that some derivative of order $k$ of $h$ is
nonzero at $x_0$. Thus the Taylor expansion of $h$ at $x_0$ has the form
\begin{equation*}
h(x_0+y)=P_k(y)+o(|y|^k),
\end{equation*}
where $P_k$ is a nonzero homogeneous polynomial of degree $k$. For $r>0$, set
\begin{equation*}
E_r=\{y\in B_1(0)):x_0+ry\in E\}.
\end{equation*}
Since $x_0$ is a density point of $E$, the measure   $$|B_1(0))\setminus E_r|\rightarrow 0$$  as $r\to 0$. On the other hand,
\begin{equation*}
\frac{h(x_0+ry)}{r^k}\longrightarrow P_k(y)
\end{equation*}
uniformly for $y\in B_1(0)$. Since $h(x_0+ry)=0$ on $E_r$, we obtain
\begin{align*}
\int_{B_1(0)} |P_k(y)|\,dy
&\le
\int_{E_r}
\left|
P_k(y)-\frac{h(x_0+ry)}{r^k}
\right|\,dy  +
\int_{B_1(0)\setminus E_r} |P_k(y)|\,dy .
\end{align*}
Letting $r\to 0$, the right-hand side tends to zero. Hence
$P_k=0$ almost everywhere on $B_1(0)$, and therefore $P_k\equiv 0$, a
contradiction. This proves the lemma.
\end{proof}

 For smooth pseudoconvex domains in  $\mathbb C^2$, Kohn \cite[Proposition 2.8]{Ko72} showed that    finite commutator type  is equivalent to  the nonflatness of the Levi determinant along tangential directions.  Since finite D'Angelo type and finite commutator type are equivalent in $\mathbb C^2$, Conjecture 2 in the introduction holds when $n=2$.  Consequently, Conjecture 1 also holds in this case by Proposition \ref{2to1}. For the reader's convenience, we include the proof below.
 
We first recall the definition of commutator type.  For a smooth real hypersurface \(M\subset \mathbb C^2\), let \(L\) be a local nonvanishing \((1,0)\)-vector field tangent to \(M\) near $p$. For \(k\ge 1\), let \(\mathcal L_k\) denote the collection of all iterated commutators of \(L\) and \(\overline L\) of length at most \(k\). The commutator type of \(M\) at \(p\) is the smallest integer \(k\) such that the values at \(p\) of the vector fields in \(\mathcal L_k\) span the complexified tangent space \(\mathbb C T_pM\). If no such \(k\) exists, the commutator type is defined to be infinite. 

 We shall also need the following Cartan's formula: if $\theta$ is s smooth one-form and $X, Y$ are smooth vector fields, then  \begin{equation} \label{Ca}\theta([X,Y]) = X(\theta(Y))-Y(\theta(X))-d\theta(X,Y). \end{equation}

\begin{proposition}\label{cn=2}
Let \(\Omega\subset\mathbb C^2\) be a smooth pseudoconvex domain of finite
D'Angelo type.  Then  the Levi
determinant, regarded as a smooth function on \(b\Omega\), does not vanish to
infinite order at any point of \(b\Omega\). In particular, the weakly pseudoconvex locus
\(W\) has zero surface measure on \(b\Omega\).   
\end{proposition}

\begin{proof}
The assertion is local on \(b\Omega\). Let \(p\in b\Omega\), and let \(\rho\) be a smooth local defining function
near \(p\).     Choose a local nonvanishing
\((1,0)\)-vector field \(L\) spanning \(T^{1,0}(b\Omega)\) near $p$.  The scalar Levi
form is then given by
\begin{equation*}
\lambda_\rho
=
\mathcal L_\rho(L,\overline L)
=
\sum_{j,k=1}^{2}
\rho_{z_j\overline z_k}L_j\overline{L_k}.
\end{equation*}
Since the Levi determinant is differed from $\lambda_\rho$  by a nonvanishing scalar factor, it suffices to show that $ \lambda_\rho $ does not vanish to infinite order at $p$ along the tangential directions $L$ and $\overline{L}$.

Let \(\theta = \frac{i}{2}(\bar\partial \rho -\partial \rho)\) be the real  contact form on \(b\Omega\). Then  \(\theta(L)=\theta(\overline L)=0\), and   \(\theta\) does not vanish on the remaining real   tangential vector field in  $  \mathbb CT(b\Omega)\big/(T^{1,0}(b\Omega)\oplus T^{0, 1}(b\Omega))$. 
\begin{comment}
    
For commutators of length two, Cartan's formula gives
\begin{equation*} d\theta(L,\overline L)=-\theta([L,\overline L]). \end{equation*} 
Since \(d\theta\) is a nonzero smooth multiple of the Levi form, we obtain \begin{equation*} \theta([L,\overline L])=a\,\lambda_\rho \end{equation*} for some smooth nonvanishing function \(a\). Thus the component of \([L,\overline L]\) in the missing real tangential direction $T$ is controlled by \(\lambda_\rho\). 
\end{comment}
  Let $C_k$ be an iterated commutator of $L$ and $\overline L$ of length $k$. We claim that
$ 
\theta(C_k)
$ 
is a finite sum of tangential derivatives of $\lambda_\rho$ of order at most $k-2$, with smooth coefficients.
For $k=2$, we have
$$
C_2=[L,\overline L].
$$
By Cartan's formula \eqref{Ca},
$$
\theta([L,\overline L])
=
-d\theta(L,\overline L).
$$
Since $d\theta(L,\overline L)$ is a nonzero smooth multiple of the Levi form, there exists a smooth nonvanishing function $a$ such that
$$
\theta(C_2)= a\lambda_\rho.
$$
Thus the claim holds for commutators of length $2$.

Assume that the claim holds for all commutators of length $k$. Let $X$ be either $L$ or $\overline L$, and consider
$$
C_{k+1}=[X,C_k].
$$
Again by Cartan's formula and the identity $\theta(X)=0$, we have
$$
\theta(C_{k+1})=\theta([X,C_k])
= X(\theta(C_k))
-d\theta(X,C_k).
$$
Choose a real tangential vector field $T$ such that $\theta(T)=1$. Then $C_k$ can be written as $$ C_k = \theta(C_k)T + c_1L + c_2\overline L, $$ where $c_1$ and $c_2$ are smooth functions. Since $d\theta$ is tensorial, we further obtain
 $$
\theta(C_{k+1}) 
 = X(\theta(C_k)) - \theta(C_k) d\theta(X, T) - c_1  d\theta (X, L) - c_2 d\theta (X, \overline{L}).
$$
 By the induction hypothesis, $\theta(C_k)$ involves tangential derivatives of $\lambda_\rho$ of order at most $k-2$. Hence $X(\theta(C_k))$ involves tangential derivatives of $\lambda_\rho$ of order at most $k-1$. The term $\theta(C_k)d\theta(X,T)$ involves no derivatives beyond those already present in $\theta(C_k)$, and therefore has order at most $k-2$ in $\lambda_\rho$. Finally, the terms $d\theta(X,L)$ and $d\theta(X,\overline L)$ are either zero or smooth multiples of the Levi form, and hence contain only the factor $\lambda_\rho$ itself. %Thus these last two terms do not introduce any higher-order derivatives of $\lambda_\rho$. 
 Altogether, $\theta(C_{k+1})$ involves tangential derivatives of $\lambda_\rho$ of order at most $k-1$, as desired. This proves  the claim.

Consequently, if $\lambda_\rho$ vanishes to infinite order at $p$ along $b\Omega$, 
then every finite iterated commutator of $L$ and $\overline L$ would have zero component in the missing real tangential direction. Hence the Lie algebra generated by $L$ and $\overline L$ would fail, at every finite order, to generate the full complexified tangent space
$ 
\mathbb C T_p(b\Omega).
$ 
Thus the commutator type at $p$ would be infinite.
 However, since  finite D'Angelo type is equivalent to finite commutator type in \(\mathbb C^2\), this   contradicts the assumption that $p$ is of finite D'Angelo type. Therefore \(\lambda_\rho\) cannot vanish to infinite order at \(p\).  
 \end{proof}

\fontsize{11}{11}\selectfont

\vspace{0.7cm}

\noindent martino.fassina@gmail.com.\\
 
\noindent pan1@pfw.edu,

\vspace{0.1 cm}

\noindent Department of Mathematical Sciences, Purdue University Fort Wayne, Fort Wayne, IN 46805-1499, USA.\\

%\vspace{0.1cm}

\noindent zhan1313@pfw.edu,

\vspace{0.1 cm}

\noindent Department of Mathematical Sciences, Purdue University Fort Wayne, Fort Wayne, IN 46805-1499, USA.\\

\end{document}